\documentclass[10pt]{amsart}

\usepackage{amsfonts,amsmath,latexsym,amssymb,verbatim,amsbsy,amsthm}
\usepackage{dsfont,bm}
\usepackage{multirow}
\usepackage{enumitem}
\usepackage{wrapfig}

\usepackage[top=1in, bottom=1in, left=1in, right=1in]{geometry}

\usepackage[dvipsnames]{xcolor} 

\usepackage{tikz}

\usepackage[colorlinks=true, pdfstartview=FitV, linkcolor=RoyalBlue, citecolor=ForestGreen, urlcolor=blue]{hyperref}

\theoremstyle{plain}
\newtheorem{THEOREM}{Theorem}[section]

\newtheorem{theorem}[THEOREM]{Theorem}

\newtheorem{lemma}[THEOREM]{Lemma}

\theoremstyle{definition}

\newtheorem{definition}[THEOREM]{Definition}

\theoremstyle{remark}

\newtheorem{remark}[THEOREM]{Remark}

\newcommand{\thm}[1]{Theorem~\ref{#1}}
\newcommand{\lem}[1]{Lemma~\ref{#1}}

\newcommand{\Z}{\ensuremath{\mathbb{Z}}}   
\newcommand{\R}{\ensuremath{\mathbb{R}}}   
\newcommand{\T}{\ensuremath{\mathbb{T}}}   
\newcommand{\cE}{\ensuremath{\mathbb{E}}}   
\newcommand{\I}{\ensuremath{\mathbb{I}}}   

\def \b {\beta}

\def \d {\delta}
\def \e {\varepsilon}
\def \f {\varphi}

\def \n {\nabla}
\def \s {\sigma}

\def \cD {\Delta}
\def \cG {\cGamma}

\def \O {\Omega}

\def \cB {\mathcal{B}}
\def \cC {\mathcal{C}}
\def \cD {\mathcal{D}}
\def \cE {\mathcal{E}}

\def \cG {\mathcal{G}}
\def \cH {\mathcal{H}}
\def \cI {\mathcal{I}}

\def \cL {\mathcal{L}}
\def \cM {\mathcal{M}}

\def \cQ {\mathcal{Q}}
\def \cR {\mathcal{R}}

\def \cV {\mathcal{V}}

\def \Rem {\mathrm{Rem}}

\def \loc {\mathrm{loc}}

\def \one {{\mathds{1}}}

\def \p {\partial}

\def \GL {Gr\"onwall's Lemma}
\def \HI {the H\"older inequality}

\def \CK{the Csisz\'ar-Kullback inequality}

\renewcommand{\geq}{\geqslant}
\renewcommand{\ge}{\geqslant}
\renewcommand{\leq}{\leqslant}
\renewcommand{\le}{\leqslant}

\def \dd  {\, \mbox{d}}

\def \da  {\, \mbox{d}a}

\def \dv  {\, \mbox{d}v}
\def \dx  {\, \mbox{d}x}

\def \dy  {\, \mbox{d}y}

\def \dzeta  {\, \mbox{d}\zeta}

\def \ds  {\, \mbox{d}s}
\def \dw  {\, \mbox{d}w}

\def \dtau  {\, \mbox{d}\tau}

\def \ddt  {\frac{\mbox{d\,\,}}{\mbox{d}t}}

\def \dd  {\mbox{d}}

\def \domain {{\O \times \R^d}}
\def \ddomain {{\O^2 \times \R^{2d}}}

\def \ut {\widetilde{u}}

\def \grad {\nabla}

\title[Isothermal nonlinear alignment limit]{Isothermal hydrodynamic limit for kinetic flocking models with nonlinear velocity alignment}

\author{Roman Shvydkoy$^\dagger$}
\address{$^\dagger$University of Illinois at Chicago, Department of Mathematics, Statistics and Computer Science, Chicago,  IL 60607, USA}
\email{shvydkoy@uic.edu}

\author{Changhui Tan$^\ddagger$}
\address{$^\ddagger$University of South Carolina, Department of Mathematics, Columbia, SC 29208, USA}
\email{tan@math.sc.edu}


\date{\today}

\begin{document}

\begin{abstract}
We study the hydrodynamic limit of kinetic flocking models with nonlinear velocity alignment and strong local Fokker--Planck forcing. In the limit of small Knudsen number the kinetic densities converge to a local Maxwellian centered around macroscopic quantities satisfying the compressible Euler system with isothermal pressure $p = \s \rho$, while macroscopic alignment is obtained by convolution of the microscopic velocity law with the standard Gaussian, resulting in a different nonlinearity. Such discrepancy has been observed already in \cite{black2025hydrodynamic}. 

In this note we rigorously justify the limit for admissible weak kinetic solutions and non-vacuous limiting macroscopic solution. Our analysis provides a quantitative rate of convergence in terms of relative entropy:  $\cH(f_\e | \mu) \lesssim \e$ provided initially $\cH(f_\e(0) | \mu(0)) \leq \e$. In the linear case we recover the known result \cite{karper2015hydrodynamic}.
\end{abstract}

\maketitle

\section{Introduction}\label{sec:introduction}

In this paper, we consider the following Vlasov--Fokker--Planck type kinetic flocking
model with nonlinear velocity alignment:
\begin{equation}\label{eq:intro-kinetic}
 \partial_t f_\e + v\cdot \grad_x f_\e
 + \grad_v\cdot\big( F(f_\e) f_\e\big)
 = \frac1\e \grad_v\cdot\big[ (v-u_\e)f_\e
 + \s \n_v f_\e \big].
\end{equation}
Here $\sigma>0$ is fixed, $f_\e=f_\e(t,x,v)$, $(t,x,v)\in \R_+\times\Omega\times\R^d$, and the spatial domain
$\Omega$ is assumed to be the torus $\T^d=[0,1]^d$ for simplicity; the whole-space case $\R^d$ is briefly discussed below.  The alignment force is
\begin{equation}\label{eq:intro-force}
 F(f)(t,x,v)=\int_{\Omega\times\R^d}
 \phi(x-x')\Phi(v'-v)f(t,x',v')\dx'\dv' .
\end{equation}
The function $\phi$ is a smooth communication kernel, while
$\Phi:\R^d\to\R^d$ describes a possibly nonlinear velocity-alignment law.  The density and momentum
associated with $f_\e$ are
\[
 \rho_\e(t,x)=\int_{\R^d}f_\e(t,x,v)\dv,
 \qquad
 \rho_\e u_\e(t,x)=\int_{\R^d}v f_\e(t,x,v)\dv.
\]

The linear choice $\Phi(z)=z$ corresponds to the classical kinetic Cucker--Smale or Vlasov-alignment model \cite{cucker2007emergent}.  A natural nonlinear generalization is the particle system
\[
 \begin{cases}
  \,\dot x_i = v_i,\\[1mm]
  \,\dot v_i = \displaystyle\frac1N\sum_{j=1}^N \phi(x_i-x_j)\Phi(v_j-v_i),
 \end{cases}
\]
with the prototypical $p$-alignment law
\begin{equation}\label{eq:intro-pPhi}
 \Phi(z)=|z|^{p-2}z,\qquad p>2.
\end{equation}
When $p=2$, the system reduces to Cucker--Smale dynamics.  For $p>2$, the nonlinearity leads to different alignment rates and different stability mechanisms; see, for instance, \cite{ha2010emergent,wen2012flocking,black2024asymptotic}.
The mean-field passage from micro to kinetic description has started in the papers of Ha and Tadmor \cite{HT2008} and Ha and Liu \cite{HL2009} for noiseless and in \cite{BCC2011} for noisy systems with linear alignment and recently cast into propagation of chaos framework by Paul and Natalini \cite{NP2021-orig} and with improved rate for heavy-tail communication in \cite{nguyen2022propagation}.  The mean-field limits for nonlinear velocity alignment started in our recent work \cite{nguyen2026mean}. We also refer to Shvydkoy's environmental-averaging framework \cite{shvydkoy2024environmental} for a broader view of alignment models across
scales.

A macroscopic representation of \eqref{eq:intro-force} is the compressible Euler
system with alignment interactions
\[
\begin{cases}
 \,\partial_t\rho+\grad_x\cdot(\rho u) = 0,\\[1mm]
 \,\partial_t(\rho u)+\grad_x\cdot(\rho u\otimes u)+\grad_x p(\rho) = \rho A_\Phi[\rho,u],
\end{cases}
\]
where
\[
 A_\Phi[\rho,u](t,x)=\int_\Omega
 \phi(x-y)\Phi(u(t,y)-u(t,x))\rho(t,y)\,\dy.
\]
For the pressureless case $p \equiv0$ and the linear law $\Phi(z)=z$, this is the Euler-alignment system.  This system has been extensively studied in recent years; see, for instance, \cite{tadmor2014critical,carrillo2016critical,shvydkoy2017eulerian,do2018global,kiselev2018global,tan2020euler,leslie2023sticky} and the book \cite{shvydkoy2021dynamics}.

The connection between the kinetic and macroscopic descriptions depends on the local equilibrium imposed by the singular relaxation.  Without velocity diffusion, the strong local relaxation enforces the monokinetic ansatz
\[
 f(t,x,v)=\rho(t,x)\delta_{v=u(t,x)}.
\]
For linear velocity alignment, the rigorous pressureless hydrodynamic limit was proved by Figalli and Kang \cite{figalli2019rigorous}.  Black and Tan \cite{black2025hydrodynamic} extended this derivation to nonlinear velocity alignment.
The new difficulty in the nonlinear case is that the first velocity moment of the force no longer closes exactly in terms of $\rho_\e$ and $u_\e$.  This produces a discrepancy term.  In the pressureless setting this discrepancy vanishes in the limit because the local equilibrium is monokinetic.

The present work concerns the isothermal analogue.  In \eqref{eq:intro-kinetic}, the right-hand side contains both local relaxation and velocity diffusion.  Therefore, the expected equilibrium is the local Maxwellian
\[
 \mu(t,x,v)=\rho(t,x)(2\pi\sigma)^{-\frac d2}
 e^{-\frac{|v-u(t,x)|^2}{2\sigma}}.
\]

The Reynolds stress then yields the pressure law
\begin{equation*}
 p(\rho)=\sigma\rho.
\end{equation*}
For linear velocity alignment this is the hydrodynamic limit studied by Karper, Mellet and Trivisa \cite{karper2015hydrodynamic}.  For nonlinear velocity alignment, however, the limiting macroscopic alignment law is no longer the original map $\Phi$.  The Maxwellian velocity fluctuations average it into
\begin{equation}\label{eq:intro-Psi}
 \Psi(z)= (4\pi\sigma)^{-\frac d2} \int_{\R^d} \Phi(z-w)
 e^{-\frac{|w|^2}{4\sigma}}\,\dw.
\end{equation}
 
Thus the expected limiting system is
\begin{equation}\label{eq:intro-limit}
\begin{cases}
 \,\partial_t\rho+\grad_x\cdot(\rho u)=0,\\[1mm]
 \,\partial_t(\rho u)+\grad_x\cdot(\rho u\otimes u)+\sigma\grad_x\rho
 =\rho A_\Psi[\rho,u].
\end{cases}
\end{equation}
For example, if $d=1$, $\sigma=1$, and $\Phi(z)=z^3$, then $\Psi(z)=z^3+6z$.  This formal observation was already made in \cite{black2025hydrodynamic}; our goal is to justify it rigorously.
The key term is the nonlinear Maxwellian discrepancy.  Let 
\[
\mu_\e = \rho_\e(t,x)(2\pi\sigma)^{-\frac d2}
 e^{-\frac{|v-u_\e(t,x)|^2}{2\sigma}}
\] 
be the local Maxwellian associated with $f_\e$, and write
\[
 \cB_\Phi[f](x)=\int_{\R^d}F(f)(x,v)f(x,v)\dv.
\]
Then the force moment is compared to its Maxwellian closure:
\[
 \cG_\e=\cB_\Phi[f_\e]-\cB_\Phi[\mu_\e]
 =\cB_\Phi[f_\e]-\rho_\e A_{\Psi}[\rho_\e,u_\e].
\]
This term is absent in the linear isothermal problem, and it is different from the pressureless discrepancy treated in \cite{black2025hydrodynamic}.  In the pressureless case the microscopic distribution concentrates at one velocity; in the isothermal case the thermal fluctuations persist and must be averaged into the limiting force.

We treat nonlinear alignment laws in a unified polynomial-growth framework. For an exponent $p\ge2$, we assume
\[
  |\Phi(z)|\le C(1+|z|^{p-1}),
  \qquad
  |\grad\Phi(z)|\le C(1+|z|^{p-2}),
\]
together with oddness, monotonicity, and coercivity relative to bounded reference velocities:
\[
  \bigl(\Phi(z)-\Phi(w)\bigr)\cdot(z-w)
  \ge c_B|z-w|^p-C_B,
  \qquad |w|\le B.
\]
The precise assumptions are stated in \textup{(H1)} below. In particular, this class contains the pure $p$-alignment law \eqref{eq:intro-pPhi} and the linear law at $p=2$. The convergence argument does not require a separate uniform higher velocity moment bound.

\begin{theorem}[Isothermal limit for nonlinear velocity alignment]
\label{thm:intro-main}
Assume \textup{(H1)--(H3)} below. Let $f_\e$ be admissible solutions of \eqref{eq:intro-kinetic}, and let $(\rho,u)$ be the smooth positive-density solution of \eqref{eq:intro-limit} in \textup{(H2)}. Then,
\[
\sup_{0\leq t \leq T} \cH(f_\e|\mu)(t)
\leq C_T\e .
\]
\end{theorem}

Let us briefly describe the proof.  The kinetic free-energy inequality provides uniform control of the free energy and of the second velocity moment, together with a strong time-integrated dissipation measuring relaxation toward the local Maxwellian.  The macroscopic variables are compared with a smooth solution of the limiting system by means of an isothermal relative entropy.

We write the relative entropy as the sum of the kinetic free energy and a macroscopic correction involving the smooth limiting density and velocity. Differentiating this correction and combining it with the kinetic free-energy inequality avoids derivatives of $u_\e$. The resulting inequality contains a stress defect, a nonnegative macroscopic alignment dissipation, and a nonlinear remainder.

The key step is to estimate the remainder together with the macroscopic alignment dissipation. We first split the spatial pairs according to the difference between their kinetic and limiting mean relative velocities. Coercivity makes the combined contribution non-positive when this difference is large. On the remaining pairs, the kinetic mean relative velocity is bounded. On these pairs, a one-sided entropy inequality exploits the upper bound on the remainder kernel and the Gaussian moments of the local Maxwellians to control the remainder by the local Maxwellian Fisher information. This yields the $O(\e)$ relative-entropy convergence rate without requiring higher velocity moments of the kinetic solution.
\bigskip

The rest of the paper is organized as follows.
Section~\ref{sec:hypotheses} introduces the hypotheses and entropy quantities and states the precise convergence result.
Section~\ref{sec:apriori-moments} derives the basic kinetic a priori estimates and the macroscopic balance laws, including the stress defect and the nonlinear Maxwellian discrepancy.
Section~\ref{sec:macro-entropy} establishes the relative-entropy inequality.
Section~\ref{sec:remainder} estimates the nonlinear remainder using coercivity, Gaussian tails, and Fisher information.
Section~\ref{sec:proof-main} completes the proof of Theorem~\ref{thm:poly-moment}.
Section~7 is devoted to the kinetic well-posedness problem.

\subsection*{Notations}
We work on the unit-volume torus $\Omega=\T^d$, with $x,x'\in\Omega$ and $v,v'\in\R^d$. Time is often suppressed. For a phase-space function $f$, we write
\[
 f=f(t,x,v),\qquad f'=f(t,x',v'),
\]
while for a spatial function $g$, we write $g=g(t,x)$ and $g'=g(t,x')$. Thus, for example, $\rho_\e'=\rho_\e(t,x')$ and $u_\e'=u_\e(t,x')$; a prime denotes evaluation at the primed variables, not differentiation. The superscript $0$ denotes evaluation at $t=0$.

We use $\nabla_x$ and $\nabla_v$ for spatial and velocity derivatives. Derivatives of a velocity law, such as $\nabla\Phi$ and $\nabla\cdot\Phi$, are taken with respect to its argument. The matrix $\I_d$ is the identity, $(a\otimes b)_{ij}=a_i b_j$, $A:B=\sum_{i,j}A_{ij}B_{ij}$, and $(\nabla_xu)_{ij}=\partial_{x_j}u_i$. The symbol $*$ denotes convolution in the indicated variables, and $\one_E$ is the indicator of a set $E$.

We denote the second velocity moment by
\[
 \cM_2(f):=\int_{\Omega\times\R^d}|v|^2f(x,v)\dx\dv.
\]
The kinetic energy is $\frac12\cM_2(f)$.
We define the weighted space
\[
 L^1_2:=L^1\bigl(\Omega\times\R^d;(1+|v|^2)\dx\dv\bigr),
 \qquad
 \|f\|_{L^1_2}:=\int_{\Omega\times\R^d}(1+|v|^2)|f(x,v)|\dx\dv.
\]
Nonnegative densities in $L^1_2$ have finite mass and second velocity moment.
We also define
\[
 L\log L:=\left\{f\in L^1(\Omega\times\R^d):
 \int_{\Omega\times\R^d}|f|\log^+|f|\dx\dv<\infty\right\},
\]
where $\log^+s=\max\{\log s,0\}$ for $s>0$, with $0\log^+0=0$.

\section{Hypotheses and statements of results}\label{sec:hypotheses}

In this section, we introduce the entropy quantities and hypotheses used throughout the paper,
followed by the precise statements of our results.  We keep the presentation close to \cite{karper2015hydrodynamic}: one first records the kinetic
notion of solution and the entropy quantities, and then states the convergence theorem
relative to a smooth solution of the limiting system.

\subsection{Kinetic flocking model}
Given a kinetic density $f_\e$, we write
\[
  \rho_\e=\int_{\R^d}f_\e\dv,
  \qquad
  \rho_\e u_\e=\int_{\R^d}v f_\e\dv.
\]
We set $u_\e=0$ on the vacuum set $\{\rho_\e=0\}$.
Introduce the entropy of a probability density $f$ relative to a probability density $\mu$ (in particular, a Maxwellian) as follows:
\[
\cH(f | \mu) = \s \int_{\Omega\times\R^d} f \log\frac{f}{\mu}\dx\dv.
\]
Recall \CK,
\[
\cH(f | \mu) \geq \frac{\s}{2}\|f-\mu\|_{L^1}^2.
\]

In the context of the hydrodynamic limit of the equation \eqref{eq:intro-kinetic} we use the relative entropy $\cH(f_\e | \mu)$ to measure the divergence of $f_\e$ from $\mu$, the local limiting Maxwellian. We note the identity
\begin{equation}\label{e:H1}
\cH(f_\e | \mu) = \cH(f_\e| \mu_\e) + \cH(\mu_\e | \mu),
\end{equation}
where the latter can be expressed solely in terms of macroscopic quantities:
\[
\cH(\mu_\e| \mu) = \frac{1}{2} \int_{\O} \rho_\e | u_\e - u|^2 \dx + \s \int_{\O}  \rho_\e \log( \rho_\e /\rho) \dx.
\]
So, if $\cH(f_\e | \mu) \to 0$, then also $\cH(\mu_\e | \mu) \to 0$, and consequently, we obtain  strong convergence of macroscopic quantities in $L^1(\O)$:
\[
\rho_\e \to \rho,\qquad
\rho_\e u_\e \to \rho u, \qquad
\rho_\e |u_\e|^2 \to \rho |u|^2.
\]

Another way to split the relative entropy, which will prove to be more convenient in the analysis of the dynamics, is to separate out the free energy from the rest of the macroscopic quantities

\[
\begin{split}
\cH(f_\e | \mu) & = \cE(f_\e) + M_\e + \frac{d}{2}\s\log(2\pi\s)\\
\cE(f_\e)  & = \int_{\domain} \left(\s f_\e \log f_\e + \frac{1}{2} |v|^2 f_\e \right)\dx \dv \\
M_\e & = \int_{\O} \left( \frac{1}{2} \rho_\e |u|^2 -  \rho_\e u_\e \cdot u - \s \rho_\e \log \rho \right) \dx.
\end{split}
\]

So, the time derivative of $\cH(f_\e | \mu)$ consists of the equation for the free energy and the macroscopic component:
\begin{equation}\label{e:relentr1}
\ddt \cH(f_\e | \mu) = \ddt \cE(f_\e) + \ddt M_\e.
\end{equation}

We introduce the modulated Fisher information:
\begin{equation}  \label{eq:J-def}
  \cI(f_\e)
  = \s^2 \int_{\Omega\times\R^d} f_\e
  \left|\grad_v \log\frac{f_\e}{\mu_\e}\right|^2 \dx \dv
   =  \int_{\Omega\times\R^d} \frac{ | \s \n_v f_\e + (v-u_\e) f_\e|^2}{f_\e} \dx \dv.
\end{equation}

For smooth solutions of \eqref{eq:intro-kinetic}, the following free-energy identity holds:
\begin{equation}\label{e:free}
  \frac{\dd}{\dd t}\cE(f_\e)
  +\frac{1}{\e}\cI(f_\e)
  +\cD_\Phi(f_\e)
  =\sigma\cQ_\Phi(f_\e),
\end{equation}
where alignment
production and entropy-production terms are given by
\[
  \cD_\Phi(f)=\frac12\iint_{\Omega^2\times\R^{2d}}
  \phi(x-x')(v-v')\cdot\Phi(v-v') f f' 
  \dx\dx'\dv\dv',
\]
\[
  \cQ_\Phi(f)= \iint_{\Omega^2\times\R^{2d}}
  \phi(x-x')\, (\grad\cdot\Phi)(v-v') f f'
  \dx\dx'\dv\dv'.
\]

This is the nonlinear analogue of the entropy identity of
\cite{karper2015hydrodynamic} for the linear isothermal model.  It motivates the entropy admissibility condition in
our weak solution class.

\begin{definition}\label{def:admissible}
We say that $f_\e$ is an \emph{admissible weak solution} of \eqref{eq:intro-kinetic} on
$[0,T]$ if $f_\e\ge0$, it is weakly continuous in time as a probability measure, it has finite free energy and second velocity moment, and $\cI(f_\e)$, $\cD_\Phi(f_\e)$, and $|\cQ_\Phi(f_\e)|$ are integrable on $[0,T]$. It satisfies
for every $\varphi\in C_c^\infty([0,T)\times\Omega\times\R^d)$
\begin{align}
\int_0^T\int_{\Omega\times\R^d}&f_\e\Big(
  \partial_t\varphi+v\cdot\grad_x \varphi+F(f_\e)\cdot\grad_v \varphi
  +\frac1\e(u_\e-v)\cdot\grad_v \varphi
  +\frac{\sigma}{\e}\Delta_v \varphi\Big)\dx\dv\,\dd t\notag\\
&\qquad
 +\int_{\Omega\times\R^d}f^0_\e(x,v)\varphi(0,x,v)\dx\dv=0,\label{eq:weaksol}
\end{align}
and, for all $t\in[0,T]$, it satisfies the entropy inequality
\begin{equation}\label{eq:admissible-entropy}
  \cE(f_\e(t))
  +\frac{1}{\e} \int_0^t \cI(f_\e) \dtau
  + \int_0^t  \cD_\Phi(f_\e) \dtau
  \le \sigma \int_0^t  \cQ_\Phi(f_\e) \dtau +   \cE(f_\e(0)).
\end{equation}
\end{definition}
This is the direct integrated inequality corresponding to the smooth identity \eqref{e:free}.  In the linear Vlasov--Fokker--Planck setting, Karper, Mellet and Trivisa construct weak solutions satisfying analogous integrated entropy inequalities \cite{karper2013existence}; while for even more general communication protocols $\phi$ including Motsch-Tadmor, construction of such solutions together with their renormalization and energy {\em equality} was established in \cite{shv-cpam}.  In the pressureless nonlinear setting, Black and Tan use weak solutions obtained by regularizing the local velocity \cite[Proposition 2.2 and Remark 2.3]{black2025hydrodynamic}.  
Throughout the paper, we work with admissible weak solutions satisfying \eqref{eq:weaksol} and \eqref{eq:admissible-entropy}.  The construction of such solutions is stated in the following theorem, which we will prove in Section~\ref{s:weak}.

\begin{theorem}\label{t:weak}
Fix $\e>0$ and $\s>0$. Assume \textup{(H1)} below for some $p\ge2$. For any nonnegative initial data $f_0\in L^1_2\cap L\log L$ with $\int_{\Omega\times\R^d}f_0\dx\dv=1$, there exists a global admissible weak solution $f_\e$ to \eqref{eq:intro-kinetic} with initial data $f_0$, in the class
\[
 f_\e\in L^\infty_\loc\bigl([0,\infty);L^1_2\cap L\log L\bigr).
\]
\end{theorem}

\subsection{Hypotheses}
We now present the hypotheses needed for the hydrodynamic limit.

\subsubsection*{\textnormal{\textbf{(H1)}} Communication protocol and velocity law.}
The communication protocol satisfies
\[
  \phi\in W^{1,\infty}(\Omega),\qquad
  \phi\ge0,\qquad
  \phi(x)=\phi(-x).
\]
Most estimates below use only the boundedness and symmetry of $\phi$.  We keep the
$W^{1,\infty}$ assumption to ensure that the macroscopic alignment force is regular
along the smooth limiting solution.

Fix an exponent $p\ge2$. The map $\Phi:\R^d\to\R^d$ is locally $C^1$, and satisfies the following assumptions:

\begin{itemize}
\item \emph{Oddness and monotonicity.}
\begin{equation}\label{H1-monotone}
  \Phi(-z)=-\Phi(z),\qquad
  \bigl(\Phi(z)-\Phi(w)\bigr)\cdot(z-w)\ge0.
  \qquad z, w\in\R^d.
\end{equation}
\item \emph{Coercivity.}
For every $B>0$ there exist constants $c_B>0$ and $C_B\ge0$ such that
\begin{equation}\label{H1-coercivity}
  \bigl(\Phi(z)-\Phi(w)\bigr)\cdot(z-w)
  \ge c_B|z-w|^p-C_B,
  \qquad |w|\le B.
\end{equation}
\item \emph{Polynomial growth.}
There is $C>0$ such that
\begin{equation}\label{H1-growth}
  |\Phi(z)|\le C(1+|z|^{p-1}),
  \qquad
  |\grad\Phi(z)|\le C(1+|z|^{p-2}).
\end{equation}
\end{itemize}

An immediate consequence of \eqref{H1-monotone} is positivity of the alignment dissipation $\cD_\Phi(f)\ge0$.
Note also that the Maxwellian-averaged alignment law $\Psi$, defined in \eqref{eq:intro-Psi}, inherits oddness and monotonicity \eqref{H1-monotone}, as well as the polynomial growth bounds \eqref{H1-growth}.  

Moreover, from the coercivity assumption \eqref{H1-coercivity}, taking $w=0$, we obtain
\begin{equation}\label{H1-Lpbound}
 |z|^p\leq \frac{1}{c_0} z\cdot \Phi(z)	+ C_0.
\end{equation}
Consequently, $\Phi$ satisfies the following condition:
for every $\delta>0$ there exists $C_\delta>0$ such that
\begin{equation}\label{H1-production}
  |\grad\cdot\Phi(z)|
  \le \delta\,z\cdot\Phi(z)+C_\delta,
  \qquad z\in\R^d.
\end{equation}
Indeed, we apply \eqref{H1-growth}, \eqref{H1-Lpbound} and Young's inequality to get
\[
 |\grad\cdot\Phi(z)| \le C(1+|z|^{p-2}) \le \delta c_0|z|^p+C_\delta \le \delta\,z\cdot\Phi(z)+C_\delta.
\]

\subsubsection*{\textnormal{\textbf{(H2)}} Smooth limiting solution.}
Let $(\rho,u)$ be a smooth solution on $[0,T]$ of the limiting system
\eqref{eq:intro-limit}.  We assume
\begin{equation}\label{H2-regularity}
  \rho>0,
  \qquad
  \rho,\ u,\ \grad_xu,\ \grad_x\log\rho\in L^\infty([0,T]\times\Omega).
\end{equation}
We also assume the total mass, which is preserved in time, to be 1, that is,
\[
\int_\Omega \rho(t,x)\dx = \int_\Omega \rho^0(x)\dx = 1.
\]
On the torus, a smooth positive density automatically has a positive lower bound on
finite time intervals. On the whole space one should not assume a uniform lower
bound, since it is incompatible with finite mass. A whole-space extension also requires appropriate spatial control to bound the negative part of the kinetic entropy and justify spatial cutoffs, for example through confinement as in \cite{karper2015hydrodynamic}. The results below are stated on the torus.

\subsubsection*{\textnormal{\textbf{(H3)}} Well-prepared initial data.}
We assume that the initial data $f_\e^0$ satisfy the admissibility assumptions
uniformly in $\e$:
\begin{equation}\label{H3-bound}
  f_\e^0\ge0,
  \qquad
  \int_{\Omega\times\R^d} f_\e^0(x,v)\dx\dv=1,
  \qquad
  \sup_{\e>0}\cE(f_\e^0)<\infty.
\end{equation}

Let $\rho_\e^0$ and $u_\e^0$ be the initial macroscopic density and
velocity associated with $f_\e^0$.  We assume that the initial data are
well prepared with respect to the limiting initial state $(\rho^0,u^0)$:
there exists $C_0>0$, independent of $\e$, such that
\begin{equation}\label{H3-full-entropy}
  \cH\bigl(
    f_\e^0\,\big|\,
    \mu^0
  \bigr)
  \le C_0\e .
\end{equation}
Using the decomposition \eqref{e:H1} at $t=0$, and the nonnegativity of
both terms in that decomposition, we obtain
\[
  \cH(f_\e^0|\mu_\e^0)\le C_0\e,
  \qquad
  \cH(\mu_\e^0|\mu^0)\le  C_0\e.
\]

\subsection{Statement of the main results}
\label{sec:main-results}

\begin{theorem}[Isothermal hydrodynamic limit for polynomial-growth alignment laws]
\label{thm:poly-moment}
Assume \textup{(H1)--(H3)} and let $\{f_\e\}_{0<\e\le1/2}$ be a family of admissible weak solutions of \eqref{eq:intro-kinetic} on $[0,T]$. Then there exists a constant $C_T>0$, independent of $\e$, such that
\[
\sup_{0\leq t \leq T}\cH\bigl(f_\e\,\big|\,\mu\bigr)(t)
+\frac{1}{4\e}\int_0^T\cI(f_\e(t))\,\dd t
\leq C_T\e.
\]
Consequently, we have
\begin{equation}\label{eq:poly-kinetic-L1-sec2}
 \sup_{0\leq t \leq T} \left\|
    f_\e(t)-\mu(t)
  \right\|_{L^1(\Omega\times\R^d)}^2
  \le
  C_T\e,
\end{equation}
and the macroscopic variables satisfy
\[
  \sup_{0\le t\le T}
  \left(
    \|\rho_\e(t)-\rho(t)\|_{L^1(\Omega)}
    +
    \|\rho_\e u_\e(t)-\rho u(t)\|_{L^1(\Omega)}
  \right)
  \le
  C_T\sqrt{\e}.
\]
\end{theorem}

\section{A priori estimates and macroscopic balance laws}
\label{sec:apriori-moments}

This section collects the kinetic estimates and balance laws that will be used in the relative entropy argument.  We first extract from the entropy admissibility inequality the uniform free-energy bound, the second velocity moment bound, and the smallness of the local Maxwellian Fisher information.  We then derive the macroscopic balance laws for $(\rho_\e,u_\e)$, identify the isothermal stress defect $\cR_\e$, and isolate the nonlinear Maxwellian discrepancy $\cG_\e$.  These estimates are the kinetic inputs for the macroscopic relative entropy inequality proved in Section~\ref{sec:macro-entropy}.

\subsection{Free-energy, moment, and Fisher-information bounds}
\label{sec:free-energy}
The purpose of this section is to extract from the entropy admissibility inequality \eqref{eq:admissible-entropy} the a priori estimate:
\begin{equation}\label{eq:Jbound}
  \int_0^T \cI(f_\e(t))\,\dd t\lesssim \e .
\end{equation}
This estimate expresses the strong relaxation toward the local Maxwellian $\mu_\e=\rho_\e \mu_{\sigma,u_\e}$.  It is the analogue of the $D_1$-control in the isothermal hydrodynamic limit of \cite{karper2015hydrodynamic}.

Note that the $O(\e)$ bound in \eqref{eq:Jbound} is not sharp for well-prepared initial data. With \eqref{H3-full-entropy} in \textup{(H3)}, Theorem~\ref{thm:poly-moment} improves it to $O\bigl(\e^2\bigr)$.

\begin{lemma}\label{l:en-bounds}
Assume \textup{(H1)}.  Let $f_\e$ be an admissible weak solution of
\eqref{eq:intro-kinetic} on $[0,T]$, with initial data $f^0_\e$ satisfying \eqref{H3-bound}.
Then there exists a constant $C$, depending on $T,d,\sigma,\|\phi\|_{L^\infty}$, the structural constants in \textup{(H1)}, and the initial free-energy bound, but independent of $\e$, such that
\[
  \sup_{0\le t\le T}\cE(f_\e(t))\le C,\qquad \sup_{0\le t\le T}\cM_2(f_\e(t))\le C,\qquad \text{and} \qquad \int_0^T\cI(f_\e(t))\,\dd t\le C\e.
\]
\end{lemma}

\begin{proof}
We start from the entropy admissibility inequality \eqref{eq:admissible-entropy}.
The only term on the right-hand side that needs to be estimated is the production
term $\cQ_\Phi$.  Applying \eqref{H1-production}, we obtain
\begin{align*}
  |\cQ_\Phi(f_\e)|
  & \le
  \iint \phi(x-y)\Big(\delta(w-v)\cdot\Phi(w-v)+C_\delta\Big)
  f_\e(x,v)f_\e(y,w)\dx\,\dy\dv\,\dw\\
  & \le 2\delta \cD_\Phi(f_\e) + C_\delta \|\phi\|_{L^\infty}.
\end{align*}
Choosing $\delta=(4\sigma)^{-1}$, we obtain
\begin{equation}
\label{eq:Q-absorbed}
  \sigma\cQ_\Phi(f_\e)
  \le \frac12\cD_\Phi(f_\e)+C.
\end{equation}
Inserting \eqref{eq:Q-absorbed} into \eqref{eq:admissible-entropy} yields, for
a.e. $t\in[0,T]$,
\begin{equation}
\label{eq:free-energy-estimated}
  \cE(f_\e(t))
  +\frac{1}{\e}\int_0^t\cI(f_\e(s))\,\ds
  +\frac12\int_0^t\cD_\Phi(f_\e(s))\,\ds
  \le \cE(f_\e^0)+Ct .
\end{equation}
Since the initial free energy is uniformly bounded by \textup{(H3)}, we immediately
obtain
\[
  \sup_{0\le t\le T}\cE(f_\e(t))\le \sup_{\e>0}\cE(f_\e^0)+CT = C_T.
\]

To estimate $\cI(f_\e)$, we use the standard lower bound for the kinetic free energy:
\[
 \cE(f_\e(t))  = \cH(f_\e|\bar{\mu}_\s)  - \frac{\sigma d}{2}\log(2\pi\sigma) \geq - \frac{\sigma d}{2}\log(2\pi\sigma),
\]
where $\bar{\mu}_\s$ denotes the global Maxwellian,
\[
\bar{\mu}_\s(v) = (2\pi\sigma)^{-\frac d2}
 e^{-\frac{|v|^2}{2\sigma}}.
\]
It then follows from \eqref{eq:free-energy-estimated} that
\[
 \int_0^t\cI(f_\e(s))\,\dd s \leq \e \Big(C_T + \frac{\sigma d}{2}|\log(2\pi\sigma)|\Big) = C\e.
\]

Finally, for the second moment, we compute
\begin{align*}
 \cE(f_\e(t)) & = \frac12 \cM_2(f_\e(t)) + \cH(f_\e| \bar{\mu}_{2\sigma}) + \sigma\int_{\Omega\times\R^d} f_\e\log \bar{\mu}_{2\sigma} \dx\dv\\
 & \ge \frac14 \cM_2(f_\e(t)) - \frac{\sigma d}{2}\log(4\pi\sigma).
\end{align*}
It implies a uniform bound:
\begin{equation}\label{e:M2}
 \cM_2(f_\e(t)) \leq 4C_T + 2\sigma d |\log(4\pi\sigma)|.
\end{equation}
\end{proof}

\subsection{Macroscopic balance laws}
\label{sec:moment-equations}

We next compute the macroscopic balance laws generated by \eqref{eq:intro-kinetic}.
The calculation is formal for smooth solutions and holds in the sense of distributions
for admissible weak solutions by testing \eqref{eq:weaksol} against functions that are
constant or linear in $v$, after the standard velocity cutoff approximation.  The
presentation follows the usual moment calculation for kinetic flocking models, with
the additional feature that the local equilibrium is Maxwellian rather than monokinetic;
compare with \cite{black2025hydrodynamic}.

Taking the zeroth velocity moment gives
\begin{equation}\label{eq:continuity-eps}
  \partial_t\rho_\e+\grad_x\cdot(\rho_\e u_\e)=0.
\end{equation}
Indeed, all velocity-divergence terms vanish after integration in $v$.

We now take the first velocity moment.  Multiplying \eqref{eq:intro-kinetic} by
$v$ and integrating in $v$, the transport term gives
\[
  \int_{\R^d}v\,v\cdot\grad_x f_\e\dv
  =
  \grad_x\cdot\int_{\R^d}v\otimes v\,f_\e\dv.
\]
The local Fokker--Planck part has zero first moment:
\[
\int_{\R^d}
 v\,\grad_v\cdot\bigl((v-u_\e)f_\e+\sigma\grad_v f_\e\bigr)\dv =
 -\int_{\R^d}(v-u_\e)f_\e\dv
 -\sigma\int_{\R^d}\grad_v f_\e\dv
 =0.
\]
The first integral vanishes by the definition of $u_\e$, and the second one
vanishes by integration by parts in velocity.

The alignment contribution is
\[
  \cB_\Phi[f_\e](x)
  =
  \iiint_{\Omega \times \R^{2d}}
  \phi(x-x')\Phi(v'-v) f_\e f_\e'\dx'\dv\dv'.
\]
Thus
\begin{equation}\label{eq:moment-raw}
  \partial_t(\rho_\e u_\e)
  +\grad_x\cdot\int_{\R^d}v\otimes v\,f_\e\dv
  =
  \cB_\Phi[f_\e].
\end{equation}

We decompose the second velocity moment into its macroscopic and microscopic parts:
\[
  \int_{\R^d}v\otimes v\,f_\e\dv
  =
  \rho_\e u_\e\otimes u_\e+\s\rho_\e\I_d+\mathcal R_\e,
\]
where the Reynolds stress tensor is defined by
\[
  \mathcal R_\e(x) = \int_{\R^d}(v-u_\e)\otimes(v-u_\e)f_\e\dv-\sigma\rho_\e\I_d.
\]
Therefore \eqref{eq:moment-raw} becomes
\begin{equation}\label{eq:moment-with-pressure}
  \partial_t(\rho_\e u_\e)
  +\grad_x\cdot(\rho_\e u_\e\otimes u_\e)
  +\sigma\grad_x\rho_\e
  =
  \cB_\Phi[f_\e]-\grad_x\cdot \cR_\e .
\end{equation}
The tensor $\cR_\e$ is the stress defect.  It measures the deviation of the
actual kinetic covariance from the Maxwellian covariance $\sigma\rho_\e\I_d$.
In the isothermal setting, the thermal part $\sigma\rho_\e\I_d$ produces the
pressure, while the remaining defect $\cR_\e$ must be controlled.

\subsection{Maxwellian closure and nonlinear discrepancy}
\label{sec:maxwellian-closure}

We now identify the closure of the alignment force on local Maxwellians. The Maxwellian force moment is
\begin{align*}
\cB_\Phi[\mu_\e](x) &=
\iiint_{\Omega \times \R^{2d}} \phi(x-x')\Phi(v'-v) \mu_\e\mu_\e'\dx'\dv\dv'.
\end{align*}
Shifting the variables $v \to v + u_\e(x)$ and $v' \to v' + u_\e(x')$, and computing the convolution of two Gaussians,  we obtain
\[
\cB_\Phi[\mu_\e](x)=
\rho_\e\int_\Omega\phi(x-x')\rho_\e'
\int_{\R^d}
\Phi\bigl(u_\e'-u_\e+a\bigr)
 \bar{\mu}_{\s} \ast \bar{\mu}_{\s}(a)\,\da\dx'.
\] 
Let us recall that $\bar{\mu}_{\s} \ast \bar{\mu}_{\s} = \bar{\mu}_{2\s}$, so
\[
\cB_\Phi[\mu_\e](x) = 
\rho_\e \int_\Omega\phi(x-x') \Psi(u_\e'-u_\e) \rho_\e'\dx' =
  \rho_\e A_{\Psi}[\rho_\e,u_\e],
\]
where $\Psi$ is defined in \eqref{eq:intro-Psi}:
\[
 \Psi(z)=\int_{\R^d}\Phi(z- \zeta) \bar{\mu}_{2\s}(\zeta)\,\dzeta.
\]

We define the nonlinear Maxwellian discrepancy by
\begin{equation}\label{eq:G-def}
  \cG_\e(x)=\cG_\Phi[f_\e](x):=\cB_\Phi[f_\e](x)-\cB_\Phi[\mu_\e](x).
\end{equation}
This discrepancy measures the failure of the nonlinear alignment force to close on the
first two velocity moments.  The hydrodynamic limit therefore requires a quantitative
estimate showing that $\cG_\e$ is small when $f_\e$ is close to the local
Maxwellian $\mu_\e$.

\begin{remark}\label{rem:discrepancy}
We emphasize that if $\Phi(z)=Lz$ is linear, then $\cG_\Phi\equiv0$.  Indeed, for
each pair $(x,y)$,
\[
 \iint_{\R^{2d}}L(w-v)f_\e(x,v)f_\e(y,w)\dv\,\dw
 =
 \rho_\e(x)\rho_\e(y)L\bigl(u_\e(y)-u_\e(x)\bigr),
\]
and the same identity holds with $\mu_\e(x,v)\mu_\e(y,w)$, since $f_\e$ and
$\mu_\e$ have the same local mass and first moment.  Thus the discrepancy is a
purely nonlinear effect.
\end{remark}

It is also useful to compare this isothermal closure with the pressureless case.
Formally, as $\sigma\to0$, the Gaussian $ \bar{\mu}_{2\s}$ converges to the Dirac
mass at the origin, and hence $\Psi \to\Phi$.  In that singular limit, the
macroscopic nonlinearity matches the kinetic nonlinearity, and the discrepancy
$\cG_\e$ reduces to the one considered in \cite{black2025hydrodynamic}.  For
fixed $\sigma>0$, however, the thermal fluctuations persist and are averaged into
the modified law $\Psi$.

Combining \eqref{eq:moment-with-pressure} and \eqref{eq:G-def}, we obtain the
closed approximate macroscopic system with two defects:
\begin{equation}\label{eq:eps-system-closed}
  \partial_t(\rho_\e u_\e)
  +\grad_x\cdot(\rho_\e u_\e\otimes u_\e)
  +\sigma\grad_x\rho_\e
  =
  \rho_\e A_{\Psi}[\rho_\e,u_\e]
  + \cG_\Phi[f_\e] 
  -\grad_x\cdot \cR_\e .
\end{equation}

\section{Dynamics of the relative entropy}\label{sec:macro-entropy}

We will now analyze the equation for the relative entropy \eqref{e:relentr1} by combining the free energy equation \eqref{e:free} with the equation for the macroscopic correction $M_\e$, which we develop next. 

Recall that the full  macroscopic system obtained in \eqref{eq:continuity-eps}
and \eqref{eq:eps-system-closed} is given by
\[
\begin{cases}
 \,\partial_t\rho_\e+\grad_x\cdot(\rho_\e u_\e)=0,\\[1mm]
 \,\partial_t(\rho_\e u_\e)
 +\grad_x\cdot(\rho_\e u_\e\otimes u_\e)
 +\sigma \grad_x\rho_\e
 =
 \rho_\e A_{\Psi}[\rho_\e,u_\e]
 + \cG_\Phi[f_\e]  - \grad_x\cdot \cR_\e .
\end{cases}
\]
The limiting system is
\[
\begin{cases}
 \,\partial_t\rho+\grad_x\cdot(\rho u)=0,\\[1mm]
 \,\partial_t(\rho u)
 +\grad_x\cdot(\rho u\otimes u)
 +\sigma\grad_x\rho
 =
 \rho A_{\Psi}[\rho,u].
\end{cases}
\]
Since $\rho>0$ and the limiting solution is smooth, we also use the
nonconservative form
\[
  (\partial_t+u\cdot\grad_x)u
  =
  A_{\Psi}[\rho,u]
  -\sigma\grad_x\log\rho .
\]

Let us compute the derivative of each component of $M_\e$. By the Reynolds Transport Theorem,
\[
\begin{split}
\ddt \frac12 \int_{\O} \rho_\e |u|^2 \dx & =\frac12 \int_{\O} \rho_\e (\p_t + u_\e \cdot \n_x) |u|^2 \dx \\
& = \int_{\O} \rho_\e \p_t u \cdot u\dx + \int_{\O} \rho_\e u_\e \cdot \n_x u \cdot u \dx \\
& = \int_\O \big( \rho_\e (u_\e - u) \cdot \n u \cdot u - \s \rho_\e  u \cdot \n \log \rho  + \rho_\e u \cdot A_{\Psi}[\rho,u]  \big) \dx \\
\end{split}
\]
Similarly, using the macroscopic systems,
\[
\begin{split}
 - \ddt \int_{\O} \rho_\e u_\e \cdot u  \dx& = \int_\O \big( -\rho_\e (u_\e - u) \cdot \n u \cdot u_\e - \s \rho_\e \n \cdot u + \s \rho_\e u_\e \cdot \n \log \rho - \n u: \cR_\e \\
& - \rho_\e u \cdot A_{\Psi}[\rho_\e,u_\e] - \rho_\e u_\e \cdot A_{\Psi}[\rho,u] -  \cG_\Phi[f_\e]  \cdot u \big) \dx, \\
- \s \ddt \int_{\O} \rho_\e \log \rho \dx &= \int_\O \big( - \s \rho_\e u_\e \cdot \n \log \rho + \s \rho_\e  u \cdot \n \log \rho +\s \rho_\e \n \cdot u \big) \dx.
\end{split}
\]
Adding these up, we can see that the pressure terms cancel out. The rest adds up to 
\begin{equation}\label{e:Me-main}
\begin{split}
 \ddt M_\e & =\int_\O\big(-\n u:\cR_\e  -  \rho_\e (u_\e - u) \cdot \n u \cdot (u_\e - u)  \big) \dx \\
 &+ \int_\O \big(   \rho_\e u \cdot (A_{\Psi}[\rho,u] - A_{\Psi}[\rho_\e,u_\e] )  - \rho_\e u_\e \cdot A_{\Psi}[\rho,u]  -  \cG_\Phi[f_\e]  \cdot u \big)  \dx.
 \end{split}
\end{equation}

Let us look at the first integral. It consists of the Reynolds term and the inertial term. The inertial term, due to the assumed regularity $\n u \in L^\infty$, is bounded by the macroscopic relative entropy $\cH(\mu_\e| \mu)$, which in turn is bounded by the main relative entropy:
\[
\left|\int_\O\rho_\e(u_\e-u)\cdot\n u\cdot(u_\e-u)\dx\right|
\leq 2\|\n u\|_{L^\infty}\cH(\mu_\e|\mu)
\leq C\cH(f_\e|\mu).
\]
The Reynolds term is bounded by the $L^1$-norm of the stress itself, which we estimate using the classical argument, see \cite[Lemma 4.8]{karper2015hydrodynamic}: observe the identity
\[
\cR_\e = \int_{\R^d}  [ 2 \s \n_v \sqrt{f_\e} + (v -u_\e) \sqrt{f_\e}] \otimes [ (v -u_\e) \sqrt{ f_\e}] \dv.
\]
By the \HI,
\[
 \int_\Omega |\cR_\e| \dx \leq \cI(f_\e)^{1/2} \sqrt{ \iint_{\O \times \R^d} |v -u_\e|^2 f_\e \dx \dv}
 \]
The internal energy that appears under the root is bounded by $\cM_2$, which by \eqref{e:M2} is uniformly bounded. We proved the bound
\[
\left| \int_\O \n u: \cR_\e \dx \right| \leq C  \cI(f_\e)^{1/2}.
\]

Now, let us examine the alignment term in \eqref{e:Me-main}: we have
\begin{align*}
& \int_\O \big(   \rho_\e u \cdot (A_{\Psi}[\rho,u] - A_{\Psi}[\rho_\e,u_\e] )  - \rho_\e u_\e \cdot A_{\Psi}[\rho,u]  -  \cG_\Phi[f_\e]  \cdot u \big)  \dx \\
 = &  \int_\O \big(   \rho_\e u \cdot (A_{\Psi}[\rho,u] - A_{\Psi}[\rho_\e,u_\e] )  - \rho_\e u_\e \cdot A_{\Psi}[\rho,u]  -  \cG_\Phi[f_\e]  \cdot (u - u_\e)  - \cG_\Phi[f_\e]  \cdot u_\e \big)  \dx \\
 = &  \int_\O \big(   \rho_\e (u_\e-u) \cdot (A_{\Psi}[\rho_\e,u_\e] - A_{\Psi}[\rho,u] ) -  \cG_\Phi[f_\e]  \cdot (u - u_\e)  - \cB_\Phi[f_\e] \cdot u_\e \big)  \dx .
\end{align*}
For the macroscopic alignment term we have
\begin{align*}
 &\int_\Omega\rho_\e (u_\e-u)\cdot
 \Bigl(
 A_{\Psi}[\rho_\e,u_\e]
 -A_{\Psi}[\rho,u]
 \Bigr)\dx
 =: -A_1+ A_2,
\end{align*}
where
\begin{align*}
 A_1 &=
 -\int_\Omega\rho_\e (u_\e-u) \cdot \Bigl( A_{\Psi}[\rho_\e,u_\e] -A_{\Psi}[\rho_\e,u] \Bigr)\dx,\\
 A_2 &=
 \int_\Omega\rho_\e (u_\e-u)\cdot \Bigl( A_{\Psi}[\rho_\e,u] -A_{\Psi}[\rho,u] \Bigr)\dx .
\end{align*}
By symmetry of $\phi$ and oddness of $\Psi$, we can express $A_1$ by
\begin{align*}
 & A_1 = \iint_{\Omega^2} \phi(x-x') A_{1,\e} (x,x')\dx\dx',\quad\text{where}\\
 & A_{1,\e} (x,x')  = 
 \frac12 \rho_\e\rho'_\e \Big(\bigl(u_\e'-u_\e\bigr)-\bigl(u'-u\bigr)\Big)\cdot \Big( \Psi\bigl(u'_\e-u_\e\bigr) - \Psi\bigl(u'-u\bigr) \Big) \geq 0.\notag
\end{align*}
To estimate $A_2$, we first notice that the kernel $\Psi$ has the same polynomial bound as $\Phi$ due to convolution with the Gaussian:
\[
  |\Psi(z)| \leq C (1 + |z|^{p-1}).
\]
Since $u\in L^\infty$,
\[
  |u(y)-u(x)|\le2\|u\|_{L^\infty},
\]
and hence
\[
  \sup_{t\in[0,T]}\sup_{x,y\in\Omega} \left| \Psi\bigl(u(t,y)-u(t,x)\bigr) \right| \le C .
\]
It follows that
\[
 \left\| A_{\Psi}[\rho_\e,u] -A_{\Psi}[\rho,u] \right\|_{L^\infty}
 \le \|\phi\|_{L^\infty} \sup_{x,y} \left|\Psi\bigl(u(y)-u(x)\bigr)\right|  \|\rho_\e-\rho\|_{L^1} \le C\|\rho_\e-\rho\|_{L^1},
\]
which by \CK, is bounded by $\cH^{1/2}(\mu_\e| \mu)$. 
Hence,
\begin{align*}
 |A_2| &\le
 C\left(\int_\Omega\rho_\e|u_\e - u|^2\dx\right)^{1/2} \|\rho_\e-\rho\|_{L^1}
 \le C\cH(\mu_\e| \mu) \leq C \cH(f_\e | \mu).
\end{align*}

Let us summarize the obtained differential inequality for the macroscopic energy $M_\e$:
\begin{equation}\label{e:Me-2}
\ddt M_\e \leq c_1 \cH(f_\e|\mu) + c_2 \sqrt{ \cI(f_\e)} - A_1 +\int_\O \big(   \cG_\Phi[f_\e] \cdot (u_\e - u)  - \cB_\Phi[f_\e] \cdot u_\e \big)  \dx.
\end{equation}
Combining \eqref{e:Me-2} with the free energy equation \eqref{e:free} we obtain
\begin{equation}\label{eq:Hfinal}
\ddt \cH(f_\e|\mu) + \frac{1}{\e}\cI(f_\e) \leq C \cH(f_\e|\mu) + C \sqrt{ \cI(f_\e)} - A_1 + \Rem_\e,
\end{equation}
where the remainder term is
\[
 \Rem_\e = -\cD_\Phi(f_\e) + \sigma\cQ_\Phi(f_\e) +\int_\O \big(   \cG_\Phi[f_\e] \cdot (u_\e - u)  - \cB_\Phi[f_\e] \cdot u_\e \big)  \dx.
\]

The rest of the proof reduces to estimation of the remainder term $\Rem_\e$.

\section{Remainder term control}\label{sec:remainder}
In this section, we focus on the control of the remainder term $\Rem_\e$.

By symmetry of $\phi$ and oddness of $\Phi$, together with the Maxwellian cancellation proved below, we express the remainder term by:
\begin{equation}\label{eq:kernel-representation}
\begin{split}
\Rem_\e &= \iint_{\Omega^2}\phi(x-x')\,\Rem_\e(x,x')\dx\dx',\\
\Rem_\e(x,x') &= \iint_{\R^{2d}} K_\Phi \bigl(x,x',v'-v\bigr)
  \bigl(f_\e f_\e'-\mu_\e\mu_\e' \bigr)\dv\dv',\\
  K_\Phi(x,x',w) &= -\frac12\bigl(w+u-u'\bigr)\cdot\Phi(w) +\s\,\grad\cdot\Phi(w).
\end{split}
\end{equation}
Define 
\begin{equation}\label{eq:delta}
  U_\e(x,x') :=(u_\e'-u_\e)-(u'-u).
\end{equation}
Using \eqref{eq:delta}, we can further decompose the kernel by
\[
  K_\Phi(x,x',w) = K_{0,\e}(x,x',w)+K_{1,\e}(x,x',w),
\]
with
\begin{align*}
  K_{0,\e}(x,x',w) 
  &:=-\frac12 \bigl(w+u_\e-u_\e'\bigr)\cdot\Phi(w) +\s\,\grad\cdot\Phi(w),  \\
  K_{1,\e}(x,x',w)
  &:=-\frac12U_\e(x,x') \cdot\Phi(w). 
\end{align*}
For the $\mu_\e\mu_\e'$-part, the $K_{0,\e}$ term vanishes against local Maxwellian, and the $K_{1,\e}$ term can be combined with $A_{1,\e}$.
\begin{lemma}\label{lem:K0K1}
The following identities hold:
\begin{equation}\label{eq:K0-Maxwellian}
  \iint_{\R^{2d}}
  K_{0,\e}(x,x',v'-v)
  \mu_\e\mu_\e'\dv\dv'
  =0,\qquad \forall\,(x,x')\in\Omega^2,
\end{equation}	
\begin{equation}\label{eq:K1-Maxwellian}
  -\iint_{\R^{2d}}K_{1,\e}(x,x',v'-v)
  \mu_\e\mu_\e'\dv\dv'-A_{1,\e}(x,x')
  =\frac12\rho_\e\rho_\e'\,U_\e(x,x')\cdot\Psi(u'-u).
\end{equation}	
\end{lemma}
\begin{proof}
For \eqref{eq:K0-Maxwellian}, we compute
\begin{align}
 &\iint_{\R^{2d}}
 K_{0,\e}(x,x',v'-v)
 \mu_\e\mu_\e'\dv\dv'
 \notag\\
 &\qquad=
 \rho_\e\rho_\e'
 \int_{\R^d}
 \left[
   -\frac12 w\cdot
   \Phi(u_\e'-u_\e+w)
   +\s\,\grad\cdot
   \Phi(u_\e'-u_\e+w)
 \right]
 \bar\mu_{2\s}(w)\,\dw.
 \label{eq:general-K0-Maxwellian}
\end{align}
Since
\[
  \grad_w \bar\mu_{2\s}(w)
  =-\frac{w}{2\s}\bar\mu_{2\s}(w),
\]
applying integration by parts yields
\begin{align*}
  \int_{\R^d}
  w \cdot\Phi(u_\e'-u_\e+w)
  \bar\mu_{2\s}(w)\,\dw
  &=2\s
  \int_{\R^d}
  \grad\cdot\Phi(u_\e'-u_\e+w)
  \bar\mu_{2\s}(w)\,\dw.
\end{align*}
Therefore the two terms in \eqref{eq:general-K0-Maxwellian} cancel exactly, and \eqref{eq:K0-Maxwellian} holds.

For \eqref{eq:K1-Maxwellian}, we compute
\begin{align*}
 &\iint_{\R^{2d}}
 K_{1,\e}(x,x',v'-v)
 \mu_\e\mu_\e'\dv\,\mathrm{d}v'
 \notag\\
 &\qquad=
 -\frac12\rho_\e\rho_\e'\,U_\e(x,x') \cdot \int_{\R^d} \Phi(u_\e'-u_\e+w)\bar\mu_{2\s}(w)\,\dw =
 -\frac12\rho_\e\rho_\e'\,
 U_\e(x,x') \cdot\Psi(u_\e'-u_\e).
\end{align*}
Recalling 
\[  A_{1,\e}(x,x')
  =\frac12\,\rho_\e\rho_\e'\,
  U_\e(x,x') \cdot
  \Bigl(
    \Psi(u_\e'-u_\e)-\Psi(u'-u)
  \Bigr).
\]
Putting the two terms together yields \eqref{eq:K1-Maxwellian}.
\end{proof}
For the $f_\e f_\e'$-part, we have
\begin{align}\label{eq:K-f}
 & \iint_{\R^{2d}}	K_\Phi(x,x',v'-v)f_\e f_\e'\dv\dv'\notag\\
 = & -\frac12\iint_{\R^{2d}}\Bigl((v'-v)-(u'-u)\Bigr) \cdot\Bigl(\Phi(v'-v)-\Phi(u'-u)\Bigr)f_\e f_\e'\dv\dv'\notag\\
 &
 -\frac12\rho_\e(x)\rho_\e(x')\,U_\e(x,x') \cdot\Phi(u'-u)
 +\s\iint_{\R^{2d}}\grad\cdot\Phi(v'-v)f_\e f_\e'\dv\dv',
\end{align}
where we have added and subtracted a term
\[
 \frac12\iint_{\R^{2d}}
 \bigl((v'-v)-(u'-u)\bigr)\cdot\Phi(u'-u) f_\e f_\e'\dv\dv' = \frac12\rho_\e\rho_\e'\,U_\e(x,x') \cdot \Phi(u'-u).
\]

Combining \eqref{eq:K0-Maxwellian}, \eqref{eq:K1-Maxwellian} and \eqref{eq:K-f}, we obtain the following representation:
\begin{align}
 \Rem_\e(x,x') - A_{1,\e}(x,x') = &
 -\frac12\iint_{\R^{2d}}\Bigl((v'-v)-(u'-u)\Bigr)\cdot\Bigl(\Phi(v'-v)-\Phi(u'-u)\Bigr)f_\e f_\e'\dv\dv'
 \notag\\
 &+\s\iint_{\R^{2d}}\grad\cdot\Phi(v'-v)f_\e f_\e'\dv\dv'
 +\frac12\rho_\e\rho_\e'\,U_\e(x,x') \cdot\Bigl(\Psi(u'-u)-\Phi(u'-u)\Bigr).
 \label{eq:Rem2}
\end{align}

The first term in \eqref{eq:Rem2} is coercive. Indeed, from \eqref{H2-regularity} we know there is a constant $B>0$ such that
\begin{equation}\label{eq:b-bounded-general}
  |u'-u|\le B \qquad \forall~ (t,x,x')\in[0,T]\times\Omega^2.
\end{equation}
Then, by \eqref{H1-coercivity}, we have
\[
  \Bigl((v'-v)-(u'-u)\Bigr)\cdot\Bigl(\Phi(v'-v)-\Phi(u'-u)\Bigr)
  \ge c_B\Bigl|(v'-v)-(u'-u)\Bigr|^p-C_B.
\]
Jensen's inequality further implies
\[
  \iint_{\R^{2d}}\Bigl|(v'-v)-(u'-u)\Bigr|^p f_\e f_\e'\dv\,\mathrm{d}v' 
  \ge \rho_\e\rho_\e'|U_\e(x,x')|^p.
\]

For the second term in \eqref{eq:Rem2}, we apply \eqref{H1-growth} and get
\[
  |\grad\cdot\Phi(v'-v)| \le C(1+|v'-v|^{p-2})
  \le \eta \Bigl|(v'-v)-(u'-u)\Bigr|^p+C_{\eta,B},
\]
for every $\eta>0$.
Choosing $\eta$ sufficiently small, the production term can be absorbed into a
fixed fraction of the first term in
\eqref{eq:Rem2}.  

Finally, for the last term in \eqref{eq:Rem2}, we have
\[
  \left|U_\e \cdot \Bigl(\Psi(u'-u)-\Phi(u'-u)\Bigr) \right|\le C_B|U_\e|.
\]

Therefore, we obtain the bound
\[
\Rem_\e(x,x')-A_{1,\e}(x,x')\le\rho_\e\rho_\e'
 \Bigl(-c|U_\e|^p+C+C|U_\e|\Bigr).
\]
Thus there exists $S>0$, depending only on the structural constants and the limiting solution bounds, such that the pairwise contribution in \eqref{eq:Rem2} is non-positive whenever
\[
  |U_\e(x,x')|\ge S.
\]
Consequently, we have
\begin{equation}\label{e:RAS}
\Rem_\e(x,x') - A_{1,\e}(x,x') \leq \Rem_\e(x,x')\one_{E_S}(x,x'), \quad E_S=\{|U_\e(x,x')|<S\}.
\end{equation}

We are left to estimate the contribution on $E_S$. 

\begin{lemma}\label{lem:one-sided-entropy}
Let $P,Q$ be probability measures such that $P\ll Q$ and
\[
 H(P\mid Q):=\int\log\frac{\dd P}{\dd Q}\,\dd P<\infty.
\]
Suppose that $k$ is a measurable function satisfying
\begin{equation}\label{eq:one-sided-entropy-assumption}
k\le M, \quad k\in L^1(P), \quad \text{and} \quad M-k\in L^2(Q),	
\end{equation}
for some constant $M$.
Then
\[
 \int k\,(\dd P-\dd Q)
 \le \sqrt{2}\,\|M-k\|_{L^2(Q)}\sqrt{H(P\mid Q)}.
\]
\end{lemma}
\begin{proof}
Set $r=\dd P/\dd Q$ and $h(r)=r\log r-r+1$, with $h(0)=1$.
Since $P$ and $Q$ have the same mass and $M-k\ge0$, we have
\begin{align*}
 \int k\,(\dd P-\dd Q) = \int k(r-1)\,\dd Q
 =\int (M-k)(1-r)\,\dd Q
 \le\int_{\{r\le1\}}(M-k)(1-r)\,\dd Q.
\end{align*}
For $r\in[0,1]$, we have $(1-r)^2\le2h(r)$. Therefore, by Cauchy--Schwarz,
\[
 \int_{\{r\le1\}}(M-k)(1-r)\,\dd Q
 \le\|M-k\|_{L^2(Q)} \left(2\int h(r)\,\dd Q\right)^{\frac12} 
 = \sqrt{2}\,\|M-k\|_{L^2(Q)}\sqrt{H(P\mid Q)},
\]
finishing the proof.
\end{proof}

We now apply Lemma \ref{lem:one-sided-entropy} to the product measures
\[
 \dd P_\e=f_\e f_\e'\dx\dx'\dv\dv',
 \qquad
 \dd Q_\e=\mu_\e\mu_\e'\dx\dx'\dv\dv'.
\]
Both measures have mass one. Moreover, $P_\e\ll Q_\e$: where $\rho_\e>0$, the local Maxwellian $\mu_\e$ is strictly positive in velocity, while on $\{\rho_\e=0\}$ we have $f_\e=0$ almost everywhere.
Applying the Gaussian logarithmic Sobolev inequality to $f_\e(x,\cdot)/\rho_\e(x)$ at each non-vacuum point $x$ and then integrating against $\rho_\e(x)\dx$, we obtain
\begin{equation}\label{eq:product-entropy-Fisher}
 H(P_\e\mid Q_\e)
 =2\int_{\Omega\times\R^d}f_\e\log\frac{f_\e}{\mu_\e}\dx\dv
 =\frac{2}{\sigma}\cH(f_\e\mid\mu_\e)
 \le\frac{1}{\sigma}\cI(f_\e).
\end{equation}

Define
\[
 k_\e(x,v,x',v')
 :=\one_{E_S}(x,x')\phi(x-x')K_\Phi(x,x',v'-v).
\]
By \eqref{eq:kernel-representation} and \eqref{e:RAS},
\begin{equation}\label{eq:entropy-remainder-reduction}
 \Rem_\e-A_1
 \le\iint_{E_S}\phi(x-x')\Rem_\e(x,x')\dx\dx'
 =\int k_\e\,(\dd P_\e-\dd Q_\e).
\end{equation}

We check the hypotheses \eqref{eq:one-sided-entropy-assumption} in Lemma~\ref{lem:one-sided-entropy}.

From the coercivity condition \eqref{H1-coercivity} on $\Phi$, together with the bound \eqref{eq:b-bounded-general} on $|u-u'|$, we have 
\[
 K_\Phi(x,x',w) \le -c|w+u-u'|^p+C_B.
\]

{
Hence, we have
\[
 k_\e\le\|\phi\|_{L^\infty} C_B=:M,
\]
where $M$ is independent of $\e$. Moreover, from \eqref{H1-growth} we have
\begin{equation}\label{eq:entropy-kernel-growth}
 |K_\Phi(x,x',z)|\le C_B(1+|z|^p).
\end{equation}
Together with \eqref{H1-Lpbound}, we obtain
\[
 \int |k_\e|\,\dd P_\e
 \le C\iint_{\Omega^2\times\R^{2d}}
       \phi(x-x')(1+|v'-v|^p)f_\e f_\e'
       \dx\dx'\dv\dv'
 \le C\bigl(1+\cD_\Phi(f_\e)\bigr)<\infty.
\]
It remains to bound the kernel in $L^2(Q_\e)$. On $E_S$, we have
\[
 |u_\e'-u_\e|\le |u'-u|+|U_\e(x,x')|\le B+S.
\]
Using the Gaussian law of the relative velocity under
$\mu_\e\mu_\e'$, together with \eqref{eq:entropy-kernel-growth},
we obtain
\begin{align*}
 \int |k_\e|^2\,\dd Q_\e
 &\le C\iint_{E_S}\rho_\e\rho_\e'
      \int_{\R^d}(1+|u_\e'-u_\e+w|^{2p})
           \bar\mu_{2\sigma}(w)\,\dw\dx\dx'\notag\\
 &\le C\iint_{E_S}\rho_\e\rho_\e'\dx\dx'
      \int_{\R^d}(1+|w|^{2p})\bar\mu_{2\sigma}(w)\,\dw
 \le C.
\end{align*}
Here we used $\int_\Omega\rho_\e\dx=1$ and finiteness of all polynomial Gaussian moments. Thus $\|M-k_\e\|_{L^2(Q_\e)}\le C$, uniformly in $\e$.

Applying Lemma~\ref{lem:one-sided-entropy} in \eqref{eq:entropy-remainder-reduction}, and then using \eqref{eq:product-entropy-Fisher}, gives
\begin{equation}\label{eq:final}
 \Rem_\e-A_1
 \le C\sqrt{H(P_\e\mid Q_\e)}
 \le C\sqrt{\cI(f_\e)}.
\end{equation}

\section{Proof of Theorem \ref{thm:poly-moment}}\label{sec:proof-main}
In this section, we finish the proof of Theorem \ref{thm:poly-moment}. For admissible weak solutions, the differential inequalities below are understood in their time-integrated form.

Insert the remainder estimate \eqref{eq:final} into \eqref{eq:Hfinal} and obtain
\[
 \ddt\cH(f_\e|\mu)
 +\frac{1}{2\e}\cI(f_\e)
 \le
 C\cH(f_\e|\mu)
 +C\sqrt{\cI(f_\e)}.
\]
By Young's inequality,
\[
 C\cI(f_\e)^{1/2}
 \le \frac{1}{4\e}\cI(f_\e) +C\e.
\]
Hence
\[
 \ddt\cH(f_\e|\mu) + \frac{1}{4\e}\cI(f_\e)
 \le C\cH(f_\e|\mu) +C\e.
\]
Applying Gronwall's inequality and by the quantitative well-preparedness assumption \eqref{H3-full-entropy}, we obtain
\[
 \sup_{0\le t\le T}\cH(f_\e|\mu)(t) + \frac{1}{4\e}\int_0^T\cI(f_\e(t))\,\dd t
 \le
 C_T\e.
\]
Finally, the decomposition
\[
  \cH(f_\e|\mu)
  =\cH(f_\e|\mu_\e)+\cH(\mu_\e|\mu)
\]
and the Csisz\'ar--Kullback estimate give \eqref{eq:poly-kinetic-L1-sec2} and the density estimate. For the momentum, we also use
\[
 \|\rho_\e u_\e-\rho u\|_{L^1}
 \le\left(\int_\Omega\rho_\e|u_\e-u|^2\dx\right)^{\frac12}
 +\|u\|_{L^\infty}\|\rho_\e-\rho\|_{L^1}.
\]
Consequently,
\[
 \sup_{0\le t\le T}
 \left(
 \|\rho_\e-\rho\|_{L^1(\Omega)}
 +\|\rho_\e u_\e-\rho u\|_{L^1(\Omega)}
 \right)
 \le
 C_T\sqrt\e.
\]

\section{Existence of admissible weak solutions} \label{s:weak}
In this section we address the question of well-posedness of the system \eqref{eq:intro-kinetic} for a fixed positive $\e$ and prove \thm{t:weak}.  Since the size of the parameters play no role in our analysis we set $\e = \s = 1$ for convenience and drop sub-index $\e$ in all of the notation:
\[
 \partial_t f + v\cdot \grad_x f
 + \grad_v\cdot\big( F(f) f \big)
 = \grad_v\cdot\big( (v-u)f
 + \n_v f \big).
\]

\subsection{Approximation scheme}

Let us fix a standard positive mollifier $\psi >0$ on $\R^d$, and for small $\d>0$ consider the periodized approximating identity 
\[
 \psi_\d(x)=\frac1{\d^d}\sum_{k\in\Z^d}\psi\Bigl(\frac{x+k}{\d}\Bigr).
\]
Note that for every fixed $\d>0$, $ \psi_\d >0$.

For the approximate solution $f_\d$, denote its moments by
\[
 \rho_\d=\int_{\R^d}f_\d\dv,\qquad
 \rho_\d u_\d=\int_{\R^d}vf_\d\dv,
\]
Let us define a special approximation of $u_\d$ based on overmollified Favre filtration:
\[
\ut_\d = \left( \frac{(u_\d\rho_\d) \ast  \psi_\d}{\rho_\d \ast  \psi_\d} \right) \ast  \psi_\d.
\]

We note that this operation is symmetric in the $\rho_\d$-weighted inner product:
\[
(\ut_\d, v)_{\rho_\d} = (u_\d, \widetilde{v}_\d)_{\rho_\d}.
\]
Many other remarkable properties of this mollification are presented in \cite{shvydkoy2021dynamics}. We just note that since $\psi_\d>0$, and $u_\d\rho_\d,\rho_\d\in L^\infty_tL_x^1$, we have $\ut_\d\in C^\infty(\Omega)$ with $|\p^k \ut_\d|\le C_{k,\d}$ uniformly on the same time interval $[0,T]$.  Moreover, Jensen's inequality gives
\begin{equation}\label{e:Favre}
 \int_\Omega\rho_\d|\ut_\d|^2\dx
 \le\int_\Omega\frac{|(u_\d\rho_\d)*\psi_\d|^2}{\rho_\d*\psi_\d}\dx
 \le\int_\Omega\rho_\d|u_\d|^2\dx.
\end{equation}

Let us fix $R=1/\d$, and consider a truncation function $\chi_R(v)=\chi(v/R)$, where $0\le\chi\le1$ is smooth, radial, nonincreasing in $|v|$, equal to one on $|v|\le1$, and zero on $|v|\ge2$. 
Let $\eta_\d$ be a standard velocity mollifier. For an initial condition $f_0 \in L^1_2 \cap L \log L$, consider the approximation
\[
 f_{0,\d}=c_\d\bigl(f_0*(\psi_\d\otimes\eta_\d)\bigr)\chi_R(v),
 \qquad \int_\domain f_{0,\d}\dx\dv=1.
\]
Here, we choose $c_\d$ such that $f_{0,\d}$ has mass 1. Such normalization is for convenience only. Since $c_\d\to1$, mollification and truncation give the uniform bounds
\[
 \cM_2(f_{0,\d})\le C,\qquad
 \int_\domain f_{0,\d}\log^+f_{0,\d}\dx\dv\le C.
\] 

Consider the Cauchy problem
\begin{equation}\label{e:appr}
 \partial_t f_\d+v\cdot\grad_x f_\d+\grad_v\cdot\big(F_\d(f_\d)f_\d\big)
 =\grad_v\cdot\big((v-\ut_\d)f_\d+\n_v f_\d\big),
 \qquad f_\d(0) =f_{0,\d},
\end{equation}
where
\[
 F_\d(f_\d)(t,x,v)
 =\int_{\Omega\times\R^d}\phi(x-x')\chi_R(v'-v)
       \Phi(v'-v)f_\d(t,x',v')\dx'\dv'.
\]
In other words, we are considering a truncated non-linearity and truncated and mollified initial data. Global existence of  solutions to \eqref{e:appr} follows from the classical fixed point argument (or linearization and iteration). For alignment models,  such approximation arguments has been addressed in \cite{shvydkoy2024environmental}. Thus, for each fixed $\d>0$, we have a global nonnegative weak solution satisfying
\[
 f_\d\in L^\infty(0,T;L^1_2\cap L^2),\qquad
 \n_v f_\d\in L^2((0,T)\times\domain)
\]
for every $T>0$, with initial trace $f_{0,\d}$, for which \eqref{e:appr} also holds in the renormalized sense.

Next, we obtain uniform-in-$\d$ a priori bounds. For the second moment, testing \eqref{e:appr} with $|v|^2/2$, we have
\begin{align*}
 \frac12\ddt\cM_2(f_\d)+\cD_{\chi_R\Phi}(f_\d)
 &=d-\cM_2(f_\d)+\int_\domain v\cdot\ut_\d f_\d\dx\dv\\
 &\le d-\cM_2(f_\d)+\cM_2(f_\d)^{\frac12}
 \left(\int_\Omega\rho_\d|\ut_\d|^2\dx\right)^{\frac12}\le d,
\end{align*}
where for the last inequality, we have used \eqref{e:Favre} and get
\[
 \int_\Omega\rho_\d|\ut_\d|^2\dx \le \int_\Omega\rho_\d|u_\d|^2\dx \le \cM_2(f_\d).
\]
Thus,
\begin{equation}\label{e:second-momunif}
 \sup_{0<\d<1}\left(
 \sup_{0\le t\le T}\cM_2(f_\d(t))
 +\int_0^T\cD_{\chi_R\Phi}(f_\d(t))\,\dd t\right)\le C_T.
\end{equation}

For the uniform $L \log L$ bound, we test \eqref{e:appr} again with $\log f_\d +1$ and obtain
\begin{align}\label{eq:flogf}
\ddt \int_{\domain} f_\d \log f_\d\dx \dv & =   \cQ_{\chi_R \Phi}(f_\d) -  \int_{\Omega\times\R^d} \frac{ ( \n_v f_\d + (v-\ut_\d) f_\d) \cdot \n_v f_\d}{f_\d} \dx \dv \notag\\
& = \cQ_{\chi_R \Phi}(f_\d) - \int_{\Omega\times\R^d} \frac{| \n_v f_\d|^2}{f_\d} \dx \dv + d.
\end{align}

For $\cQ_{\chi_R \Phi}(f_\d)$, we apply \eqref{H1-production} and obtain the following bound analogous to \eqref{eq:Q-absorbed}
\begin{align}\label{eq:Q-absorbed-R}
 \cQ_{\chi_R\Phi}(f_\d)\notag
 &=\iint_\ddomain\phi(x-x')\chi_R(v'-v)(\nabla\cdot\Phi)(v'-v)f_\d f_\d'
       \dx\dx'\dv\dv'\notag\\
 &\quad+\iint_\ddomain\phi(x-x')\nabla\chi_R(v'-v)\cdot\Phi(v'-v)f_\d f_\d'
       \dx\dx'\dv\dv'\notag\\
 &\le\frac12\cD_{\chi_R\Phi}(f_\d)+C
       +\iint_\ddomain\phi(x-x')\nabla\chi_R(v'-v)\cdot\Phi(v'-v)f_\d f_\d'
       \dx\dx'\dv\dv'\notag\\
 &=\frac12\cD_{\chi_R\Phi}(f_\d)+C - B_{R,\Phi}(f_\d),
\end{align}
where the constant $C$ is independent of $R$ and $\d$, and
\[
 B_{R,\Phi}(f_\d):= \iint_\ddomain\phi(x-x')\bigl(-\nabla\chi_R(v'-v)\cdot\Phi(v'-v)\bigr) f_\d f_\d'\dx\dx'\dv\dv'\geq0,
\]
since $\chi_R$ is radially decreasing and $(v'-v)\cdot\Phi(v'-v)\ge0$. 
The negative part of the entropy is bounded by
\[
 -\int_{\{0<f_\d<1\}}f_\d\log f_\d\dx\dv
 \le\int_\domain |v|^2f_\d\dx\dv
 +\int_\domain e^{-1-|v|^2}\dx\dv
 \le \cM_2(f_\d)+C.
\]

Integrating \eqref{eq:flogf} in time, and together with \eqref{e:second-momunif}, this yields
\begin{equation}\label{e:second-entunif}
 \sup_{0<\d<1}\left(\sup_{0\le t\le T}\int_\domain f_\d\log^+f_\d\dx\dv
 +\int_0^T\int_\domain\frac{|\n_v f_\d|^2}{f_\d}\dx\dv\,\dd t + \int_0^T B_{R,\Phi}(f_\d)\,\dd t\right)\le C_T.
\end{equation}

We also obtain absolute-value bounds on the force and its divergence. For $F_\delta$, we have
\[
 \int_\domain f_\d|F_\d(f_\d)|\dx\dv
 \le\iint_\ddomain\phi(x-x')\chi_R(v'-v)|\Phi(v'-v)|f_\d f_\d'
       \dx\dx'\dv\dv'\\
 \le C\bigl(1+\cD_{\chi_R\Phi}(f_\d)\bigr),
\]
where we have used \eqref{H1-growth} and \eqref{H1-Lpbound} to get
 \[
|\Phi(z)|\le C(1+|z|^{p-1}) \le C(1+|z|^p) \le C(1+z\cdot\Phi(z)).
\]
For $\nabla_v\cdot F_\delta$, a similar calculation as in \eqref{eq:Q-absorbed-R} gives
\begin{align*}
 \int_\domain f_\d|\nabla_v\cdot F_\d(f_\d)|\dx\dv
 &\le\iint_\ddomain\phi(x-x')\bigl|\nabla\cdot(\chi_R\Phi)(v'-v)\bigr|
       f_\d f_\d'\dx\dx'\dv\dv'\\
 &\le\iint_\ddomain\phi(x-x')\chi_R(v'-v)|\nabla\cdot\Phi(v'-v)|
       f_\d f_\d'\dx\dx'\dv\dv'\\
 &\quad+\iint_\ddomain\phi(x-x')\bigl(-\nabla\chi_R(v'-v)\cdot\Phi(v'-v)\bigr)
       f_\d f_\d'\dx\dx'\dv\dv'\\
 &\le \frac12\cD_{\chi_R\Phi}(f_\d) + C + B_{R,\Phi}(f_\d).
\end{align*}
Consequently, \eqref{e:second-momunif} and \eqref{e:second-entunif} imply
\begin{equation}\label{e:second-force-bound}
 \sup_{0<\d<1}\int_0^T\int_\domain
 f_\d\bigl(|F_\d(f_\d)|+|\nabla_v\cdot F_\d(f_\d)|\bigr)
                  \dx\dv\,\dd t\le C_T.
\end{equation}

Since $f_\d$ has no uniform $L^q$-bound for $q>1$, we fix a small renormalization parameter $\epsilon>0$, and consider the renormalizer 
\[\b(x) = \frac{x}{1 + \epsilon x}.\]
 Note that $\b(f_\d)$ is uniformly bounded in $L^q$ with respect to $\d$ for $q\ge1$. Indeed, since $0\le\b(s)\le\min\{s,\epsilon^{-1}\}$,
we have
\[
 \|\b(f_\d)\|_{L^q([0,T]\times\domain)}^q
 \le\epsilon^{1-q}\int_0^T\int_\domain f_\d\dx\dv\,\dd t
 =T\epsilon^{1-q}.
\]

We now write the equation for $\b(f_\d)$ testing  \eqref{e:appr} with $\b'(f_\d)$. All the derivatives are classical at this stage, so we directly obtain
\begin{equation*}
\begin{split}
\p_t \b(f_\d) + v\cdot \grad_x \b(f_\d)
 + \grad_v\cdot\big( \b(f_\d)  F_\d(f_\d)  \big) + \big( \b'(f_\d)  f_\d - \b(f_\d) \big) \nabla_v\cdot F_\d(f_\d)
 =  \b'(f_\d) \grad_v\cdot\big( (v-\ut_\d)f_\d
 + \n_v f_\d \big).
\end{split}
\end{equation*}
We rewrite the right hand side as 
\[
 \b'(f_\d) \grad_v\cdot\big( (v-\ut_\d)f_\d
 + \n_v f_\d \big) = d \b'(f_\d) f_\d+\b'(f_\d) (v - \ut_\d) \cdot \n_v f_\d+  \Delta_v  \b(f_\d) -  \b''(f_\d) | \n_v f_\d|^2.
\]
Thus, the renormalized solution satisfies
\[
\p_t \b(f_\d) + v\cdot \grad_x \b(f_\d) = H_1 + \n_v\cdot H_2 + \Delta_v H_3,
\]
where 
\begin{equation*}
\begin{split}
H_1& =  - ( \b'(f_\d)  f_\d - \b(f_\d) ) \nabla_v\cdot F_\d(f_\d)+d \b'(f_\d) f_\d+\b'(f_\d) (v - \ut_\d) \cdot \n_v f_\d-  \b''(f_\d) | \n_v f_\d|^2\\
H_2 & = - \b(f_\d)  F_\d(f_\d)  \\
H_3 & = \b(f_\d).
\end{split}
\end{equation*}
With $\epsilon$ fixed we consider this as a family parameterized by $\d$. Observe that
\[
 \b'(s)=\frac1{(1+\epsilon s)^2},\qquad
 \b''(s)=-\frac{2\epsilon}{(1+\epsilon s)^3},\qquad
 \b(s)\le s,\qquad |s\b'(s)-\b(s)|\le s.
\]
By \eqref{e:second-force-bound}, $H_2,H_3\in L^1([0,T]\times\domain)$, and the first two terms of $H_1$ are also uniformly bounded in this space. Now,
\[
 |\b'(f_\d)(v-\ut_\d)\cdot\n_v f_\d|
 \le |v|^2f_\d+|\ut_\d|^2f_\d+\frac{|\n_v f_\d|^2}{f_\d}
 \in L^1([0,T]\times\domain)
\]
uniformly in $\d$, by \eqref{e:Favre}, \eqref{e:second-momunif}, and \eqref{e:second-entunif}. Also,
\[
 |\b''(f_\d)|\,|\n_v f_\d|^2
 =\frac{2\epsilon f_\d}{(1+\epsilon f_\d)^3}
       \frac{|\n_v f_\d|^2}{f_\d}
 \le2\frac{|\n_v f_\d|^2}{f_\d}
 \in L^1([0,T]\times\domain)
\]
uniformly in $\d$. So, $H_1 \in L^1([0,T] \times \domain)$ as well.

\subsection{Obtaining a weak solution}\label{sec:wp-limit}
By the velocity averaging lemma in \cite{DLM1991} and the fact that $f_\d \in L \log L$ we conclude that there exists a subsequence $\d_n \to 0$ such that, denoting $f_n=f_{\d_n}$, $\rho_n=\rho_{\d_n}$, $u_n=u_{\d_n}$, and $\ut_n=\ut_{\d_n}$,

\begin{itemize}
\item $f_n \to f \in L^1([0,T] \times \domain)$ weakly,
\item $\displaystyle\int_{\R^d} \b(f_n)(t,x,v) \psi(v) \dv \to  \int_{\R^d} \b(f)(t,x,v) \psi(v) \dv$ strongly in $L^1([0,T] \times \O)$ for every $\psi \in C^\infty_0(\R^d)$. 	
\end{itemize}
By picking a sequence $\epsilon_k \to 0$ and using the diagonalization we can assume that the latter limit holds not just for one $\epsilon$, but for all $\epsilon_k$ in the sequence. Arguing as in \cite{Choi-Song2025} we see that 
\[
\sup_n \{ \| \b(f_n) - f_n \|_{ L^1([0,T] \times \domain) } \} + \| \b(f) - f \|_{ L^1([0,T] \times \domain) } \to 0
\]
as $k \to \infty$. Consequently,

\begin{itemize}
\item $\displaystyle\int_{\R^d} f_n(t,x,v) \psi(v) \dv \to  \int_{\R^d} f(t,x,v) \psi(v) \dv$ strongly in $L^1([0,T] \times \R^d)$ for every $\psi \in C^\infty_0(\R^d)$.	
\end{itemize}
With that, we get into the range of application of \cite[Lemma 2.3]{Choi-Song2025} to conclude that in fact
\begin{itemize}
 \item $f_n \to f \in L^1([0,T] \times \domain)$ strongly and almost everywhere. 	
\end{itemize}

Since $f_n$'s solve
\begin{equation}\label{e:appr2}
\partial_t f_n + v\cdot \grad_x f_n
 + \grad_v\cdot\big( F_n(f_n) f_n \big)
 = \grad_v\cdot\big( (v-\ut_n)f_n
 + \n_v f_n \big)
\end{equation}
classically, we only have to check that all the terms under the gradients converge to their natural limits weakly. We only check the non-trivial limits $F_n(f_n) f_n \to F(f) f$ and $(v-\ut_n)f_n \to (v-u)f$.

Let us first consider the alignment force. To estimate its difference, we fix $L\ge1$ and split the pair integral into the regions $|v'-v|<L$ and $|v'-v|\ge L$. Writing $R_n=1/\d_n$, we have
\begin{align}\label{eq:Fsplit}
 &\|F_n(f_n)f_n-F(f)f\|_{L^1([0,T]\times\domain)}\notag\\
 &\le\int_0^T\int_{\Omega^2}\int_{|v'-v|<L}
 \phi(x-x')|\Phi(v'-v)|
 \big|\chi_{R_n}(v'-v)f_nf_n'-ff'\big|
 \dv\dv'\dx\dx'\,\dd t\notag\\
 &\quad+\int_0^T\int_{\Omega^2}\int_{|v'-v|\ge L}
 \phi(x-x')|\Phi(v'-v)|
 \big|\chi_{R_n}(v'-v)f_nf_n'-ff'\big|
 \dv\dv'\dx\dx'\,\dd t\notag\\
 &=:I_n^{<L}+I_n^{\ge L}.
\end{align}

We first estimate the long-range contribution $I_n^{\ge L}$. 
The polynomial growth assumption \eqref{H1-growth} implies
\[
 |\Phi(z)|\le C|z|^{p-1}\le\frac{C}{L}|z|^p.
\]
Then, we have
\[
 I_n^{\ge L} \leq \frac{C}{L}\int_0^T\int_{\Omega^2}\int_{|v'-v|\ge L}\phi(x-x')|v'-v|^p \bigl(\chi_{R_n}(v'-v)f_nf_n'+ff'\bigr) \dv\dv'\dx\dx'\,\dd t
\]
By \eqref{H1-Lpbound} and the dissipation bound \eqref{e:second-momunif}, we obtain
\begin{equation}\label{e:wp-relative-moment}
 \int_0^T\iint_\ddomain\phi(x-x')\chi_{R_n}(v'-v)|v'-v|^p f_nf_n'
                \dx\dx'\dv\dv'\,\dd t
 \le C\int_0^T\cD_{\chi_{R_n}\Phi}(f_n)\,\dd t+CT\le C_T.
\end{equation}
By Fatou's lemma, the same estimate holds for $ff'$ without the cutoff.  Consequently, the long-range contribution is bounded by
\[
 I_n^{\ge L} \le \frac{C_T}{L}.
\]

Now let us turn to the short-range contribution $I_n^{<L}$. For fixed $L$, we have $\chi_{R_n}(v'-v)=1$ on $|v'-v|<L$ for all sufficiently large $n$. By \eqref{H1-growth}, the kernel is bounded on this region by
\[
 C_L:=\|\phi\|_{L^\infty}\sup_{|v'-v|<L}|\Phi(v'-v)|
 \le C(1+L^{p-1})\le CL^{p-1},
\]
since $L\ge1$. Using $f_nf_n'-ff'=f_n(f_n'-f')+(f_n-f)f'$, we obtain
\begin{align*}
 I_n^{<L}
 &\le C_L\|f_nf_n'-ff'\|_{L^1([0,T]\times\ddomain)}
 \le C_L\int_0^T\iint_\ddomain
 \bigl(f_n|f_n'-f'|+|f_n-f|f'\bigr)\dx\dx'\dv\dv'\,\dd t\\
 &=2C_L\|f_n-f\|_{L^1([0,T]\times\domain)}\to0.
\end{align*}
Here we have used mass conservation and the strong $L^1$ convergence of $f_n$.

Combining the short- and long-range estimates and first letting $n\to\infty$ with $L$ fixed, we obtain
\[
 \limsup_{n\to\infty}
 \|F_n(f_n)f_n-F(f)f\|_{L^1([0,T]\times\domain)}
 \le\frac{C_T}{L}.
\]
Letting $L\to\infty$ then proves $F_n(f_n)f_n\to F(f)f$ strongly in $L^1([0,T]\times\domain)$.

Next, the strong limit 
\[
v f_n \to v f
\]
holds in $L^1([0,T]\times\domain)$. Indeed, the second moment bound \eqref{e:second-momunif}, which passes to $f$ by Fatou's lemma, gives, 
\begin{align*}
 \|vf_n-vf\|_{L^1([0,T]\times\domain)}
 &\le\left(\int_0^T\int_\domain |v|^2|f_n-f|\dx\dv\,\dd t\right)^{\frac12} \|f_n-f\|_{L^1([0,T]\times\domain)}^{\frac12}\\
 &\le\left(\int_0^T\int_\domain |v|^2(f_n+f)\dx\dv\,\dd t\right)^{\frac12} \|f_n-f\|_{L^1([0,T]\times\domain)}^{\frac12}\\
 &\le C_T\|f_n-f\|_{L^1([0,T]\times\domain)}^{\frac12}\to0.
\end{align*}

Finally, we prove the distributional limit 
\[
\ut_n f_n \to u f.
\]
To this end, by approximation, it is sufficient to assume that the test-function is given by $\f(t,x,v) = \chi(t,x) \eta(v)$. Denoting $\bar{\rho}_n = \int_{\R^d} f_n(t,x,v) \eta(v) \dv$, we therefore reduce the limit to showing 
\begin{equation}\label{e:second-weakl}
\ut_n \bar{\rho}_n \to u \bar{\rho}, 
\end{equation}
weakly. 
By the strong $L^1$ convergence of $f_n$, we already know that $\bar{\rho}_n \to \bar{\rho}$ strongly in $L^1$ and almost everywhere. At the same time, both $u_n \rho_n \to u\rho$ and $\rho_n \to \rho$ strongly in $L^1$ and almost everywhere by the strong $L^1$ convergence of $vf_n$ and $f_n$, respectively. Moreover, since  
\begin{equation*}
\begin{split}
\| (u_n\rho_n) \ast  \psi_{\d_n} - (u\rho) \ast  \psi_{\d_n}\|_{L^1} \leq \| u_n \rho_n - u\rho\|_{L^1} & \to 0, \\
\| \rho_n \ast  \psi_{\d_n} - \rho  \ast  \psi_{\d_n} \|_{L^1} \leq \| \rho_n -  \rho \|_{L^1} & \to 0, \\
\| (\bar{\rho}_n \chi) \ast  \psi_{\d_n} - (\bar{\rho} \chi) \ast  \psi_{\d_n} \|_{L^1} \leq \| \bar{\rho}_n \chi - \bar{ \rho}\chi \|_{L^1} & \to 0,
\end{split}
\end{equation*}
by passing to a further subsequence, we can assume that
\begin{equation*}
\begin{split}
|(u_n\rho_n) \ast  \psi_{\d_n} - (u\rho) \ast  \psi_{\d_n}| & \to 0, \\
| \rho_n \ast  \psi_{\d_n} - \rho \ast  \psi_{\d_n}|& \to 0, \\
| (\bar{\rho}_n \chi) \ast  \psi_{\d_n} -(\bar{\rho} \chi)  \ast  \psi_{\d_n}| &\to 0
\end{split}
\end{equation*}
almost everywhere too for all  $\chi\in C^\infty$. Consequently,
\begin{equation}\label{e:second-pointwise}
\begin{split}
|(u_n\rho_n) \ast  \psi_{\d_n} - u\rho| &\to 0, \\
| \rho_n \ast  \psi_{\d_n} - \rho | & \to 0, \\
| (\bar{\rho}_n \chi) \ast  \psi_{\d_n} - \bar{\rho} \chi | & \to 0
\end{split}
\end{equation}
almost everywhere.

 Transferring the convolution with $\psi_\d$ onto $\bar{\rho}_n \chi$ the limit \eqref{e:second-weakl} amounts to showing that
\begin{equation}\label{e:second-ultlimit}
\int_\O \frac{(u_n\rho_n) \ast  \psi_{\d_n}}{\rho_n \ast  \psi_{\d_n}} ( \bar{\rho}_n \chi) \ast  \psi_{\d_n} \dx \to \int_\O u \bar{\rho} \chi \dx.
\end{equation}
First, we have a pointwise almost everywhere convergence of the integrands. Indeed, if $\rho_n \ast  \psi_{\d_n}(x) \to \rho(x) \neq 0$, then the convergence at $x$ follows from \eqref{e:second-pointwise}. If $\rho_n \ast  \psi_{\d_n}(x) \to \rho(x) = 0$, then since $\bar{\rho}_n \leq C \rho_n$ and $\bar{\rho} \leq C \rho$,  we have $u(x) \bar{\rho}(x) \chi(x) = 0$, and
\[
 \left|  \frac{(u_n\rho_n) \ast  \psi_{\d_n}}{\rho_n \ast  \psi_{\d_n}} ( \bar{\rho}_n \chi) \ast  \psi_{\d_n}  \right| \leq C | (u_n\rho_n) \ast  \psi_{\d_n}| \to |u\rho| = 0.
 \]
Moreover, since $| (u_n\rho_n) \ast  \psi_{\d_n}|  \to  |u\rho| $ in $L^1$, the generalized Lebesgue Dominated Convergence Theorem applies to imply \eqref{e:second-ultlimit}.

\subsection{Entropy inequality}
To obtain the entropy inequality \eqref{eq:admissible-entropy},  we test the approximate equation \eqref{e:appr2} with $\log f_n+|v|^2/2$ and obtain
\[
 \frac{\dd}{\dd t}\cE(f_n)+\cD_{\chi_{R_n}\Phi}(f_n)
 +\iint_\domain\frac{|\nabla_v f_n|^2}{f_n}\dx\dv
 +\cM_2(f_n)-2d-\int_\Omega\rho_nu_n\cdot\ut_n\dx
 =\cQ_{\chi_{R_n}\Phi}(f_n).
\]
By the definition \eqref{eq:J-def},
\[
 \cI(f_n)=\iint_\domain\frac{|\nabla_v f_n|^2}{f_n}\dx\dv
 +\cM_2(f_n)-2d-\int_\Omega\rho_n|u_n|^2\dx.
\]
Thus,
\[
 \frac{\dd}{\dd t}\cE(f_n)+\cD_{\chi_{R_n}\Phi}(f_n)+\cI(f_n)
  =\cQ_{\chi_{R_n}\Phi}(f_n) +\int_\Omega\rho_nu_n\cdot(\ut_n-u_n)\dx \leq \cQ_{\chi_{R_n}\Phi}(f_n),
\]
where for the last inequality, we apply \eqref{e:Favre} and get
\[
 \int_\Omega\rho_nu_n\cdot\ut_n\dx
 \le\left(\int_\Omega\rho_n|u_n|^2\dx\right)^{1/2}
 \left(\int_\Omega\rho_n|\ut_n|^2\dx\right)^{1/2}
 \le\int_\Omega\rho_n|u_n|^2\dx.
\]
Integrating in time, we obtain the approximate entropy inequality
\begin{equation}\label{eq:admissible-entropy-n}
  \cE(f_n(t))
  + \int_0^t \cI(f_n) \dtau
  + \int_0^t \cD_{\chi_{R_n}\Phi}(f_n) \dtau
  \le \int_0^t \cQ_{\chi_{R_n}\Phi}(f_n) \dtau + \cE(f_n(0)).
\end{equation}
By the lower semicontinuity of the free energy, and passing to a further subsequence $n_k$ (depending on time $t$) we obtain $\liminf_{k\to \infty}  \cE(f_{n_k}(t)) \geq  \cE(f(t))$. Obviously 
 $\cE(f_n(0)) \to \cE(f(0))$. 

Turning to the other terms, we first consider $\cQ$. 
 By the decomposition in \eqref{eq:Q-absorbed-R}, we have
\[
\cQ_{\chi_{R_n}\Phi}(f_n)+B_{R_n,\Phi}(f_n)
=\iint\phi(x-x')\chi_{R_n}(v'-v) (\nabla\cdot\Phi)(v'-v)f_nf_n'\dx\dx'\dv\dv'.
\]
By definition,
\[
 \cQ_\Phi(f)=\iint_\ddomain\phi(x-x')(\nabla\cdot\Phi)(v'-v)ff'
 \dx\dx'\dv\dv'.
\]
Subtracting these expressions, we obtain
\begin{align*}
 &\cQ_{\chi_{R_n}\Phi}(f_n)+B_{R_n,\Phi}(f_n)-\cQ_\Phi(f)\\
 &=\iint_\ddomain\phi(x-x')\chi_{R_n}(v'-v)(\nabla\cdot\Phi)(v'-v)
 (f_nf_n'-ff')\dx\dx'\dv\dv'\\
 &\quad+\iint_\ddomain\phi(x-x')\bigl(\chi_{R_n}(v'-v)-1\bigr)
 (\nabla\cdot\Phi)(v'-v)ff'\dx\dx'\dv\dv'.
\end{align*}

For the first term, we use the same splitting as in \eqref{eq:Fsplit}, with $L\ge1$. By \eqref{H1-growth}, we have
\[
 |\nabla\cdot\Phi(v'-v)|\le C(1+|v'-v|^{p-2})
 \le\begin{cases}
  C L^{p-2}, & |v'-v|<L\\
  \frac{C}{L^2}|v'-v|^p, & |v'-v|\ge L.	
 \end{cases}
\]
Using \eqref{e:wp-relative-moment}, its counterpart for $ff'$, and mass conservation, we obtain
\begin{align*}
 &\int_0^t\left|\iint_\ddomain\phi(x-x')\chi_{R_n}(v'-v)
 (\nabla\cdot\Phi)(v'-v)(f_nf_n'-ff')\dx\dx'\dv\dv'\right|\dtau\\
 &\le CL^{p-2}\|f_nf_n'-ff'\|_{L^1([0,T]\times\ddomain)}\\
 &\quad+\frac{C}{L^2}\int_0^t\iint_\ddomain\phi(x-x')|v'-v|^p
 \bigl(\chi_{R_n}(v'-v)f_nf_n'+ff'\bigr)\dx\dx'\dv\dv'\dtau\\
 &\le CL^{p-2}\|f_n-f\|_{L^1([0,T]\times\domain)}+\frac{C_T}{L^2}.
\end{align*}
For fixed $L$, the first term on the right tends to zero by the strong $L^1$ convergence of $f_n$. Letting $L\to\infty$ then shows that the contribution of the first integral vanishes.

For the second term, $\chi_{R_n}(v'-v)-1$ vanishes on $|v'-v|\le R_n$. The same growth bound therefore gives
\begin{align*}
 &\int_0^t\left|\iint_\ddomain\phi(x-x')\bigl(\chi_{R_n}(v'-v)-1\bigr)
 (\nabla\cdot\Phi)(v'-v)ff'\dx\dx'\dv\dv'\right|\dtau\\
 &\le\frac{C}{R_n^2}\int_0^t\int_{\Omega^2}\int_{|v'-v|>R_n}
 \phi(x-x')|v'-v|^pff'\dv\dv'\dx\dx'\dtau
 \le\frac{C_T}{R_n^2}\to0.
\end{align*}

Combining the two estimates and using $B_{R_n,\Phi}(f_n)\ge0$, we obtain
\[
 \limsup_{n\to\infty}\int_0^t\cQ_{\chi_{R_n}\Phi}(f_n)\dtau
 \le\int_0^t\cQ_\Phi(f)\dtau.
\]

Next, for the alignment dissipation term $\cD$, Fatou's lemma directly gives
\[
 \int_0^t\cD_\Phi(f)\dtau
 \le\liminf_{n\to\infty}\int_0^t\cD_{\chi_{R_n}\Phi}(f_n)\dtau.
\]

As to the Fisher information term, we have
\[
 \int_0^t \cI(f_n) \dtau =  \int_0^t   \int_{\Omega\times\R^d} \frac{ | \n_v f_n|^2}{f_n} \dv \dx \dtau +  \int_0^t \int_{\Omega\times\R^d} |v-u_n|^2 f_n \dv\dx \dtau - 2dt.
\]
The first term is the classical Fisher information which is weakly lower semi-continuous under weak $L^1((0,t)\times\domain)$ convergence
\[
\liminf_{n\to \infty} \int_0^t   \int_{\Omega\times\R^d} \frac{ | \n_v f_n|^2}{f_n} \dv \dx \,\dd\tau \geq  \int_0^t   \int_{\Omega\times\R^d} \frac{ | \n_v f |^2}{f} \dv \dx \,\dd\tau.
\]

For the second term, take a subsequence realizing the lower limit. The strong $L^1$ convergence of $f_n$ and $vf_n$ gives, after further extraction, $f_n\to f$ almost everywhere and $u_n\to u$ almost everywhere on $\{\rho>0\}$. Since $f=0$ on $\{\rho=0\}$, Fatou's lemma yields
\[
 \liminf_{n\to\infty}\int_0^t\iint_\domain |v-u_n|^2f_n\dx\dv\dtau
 \ge \int_0^t\iint_\domain|v-u|^2f\dx\dv\dtau.
\]
Combining the two inequalities and subtracting $2dt$, we obtain
\[
 \liminf_{n\to\infty}\int_0^t\cI(f_n)\dtau \ge \int_0^t\cI(f)\dtau.
\]
This proves that the approximate entropy equality \eqref{eq:admissible-entropy-n} in the limit passes to the inequality \eqref{eq:admissible-entropy}.

\bigskip
\subsection*{Acknowledgment.}  
The work of R. Shvydkoy was supported in part by NSF grant DMS-2405326 and the Simons Foundation. 
The work of C. Tan was supported in part by NSF grant DMS-2238219.

The authors are especially grateful to Young-Pil Choi for valuable discussions and for communicating the one-sided entropy inequality in Lemma~\ref{lem:one-sided-entropy} and its application to the alignment remainder. His suggestion improved the relative-entropy convergence rate in an earlier version of the paper from $O\bigl(\e(1+|\log\e|^{p-2})\bigr)$ to $O(\e)$.

\bigskip

\bibliographystyle{plain}

\end{document}